\documentclass[a4paper,12pt]{amsart}
\usepackage{amsmath}
\usepackage{amssymb}
\usepackage{amsthm,amsxtra}
\usepackage{latexsym, mathrsfs, color, soul}
\usepackage{cite}
\usepackage{enumerate}
\usepackage{enumitem}
\usepackage[mathscr]{eucal}
\usepackage{multicol}
\usepackage{multirow}
\usepackage{adjustbox}
\usepackage{float}
\usepackage{graphicx}
\usepackage{mathtools}
\usepackage[hyphens]{url}
\usepackage[colorlinks=true]{hyperref}
\usepackage{cleveref}
\usepackage{tikz-cd}
\usepackage{tkz-graph}
\usepackage{tkz-berge}
\usetikzlibrary{positioning,arrows,shapes.geometric,trees,decorations.text}
\usepackage[T1]{fontenc}
\usepackage{lmodern}
\usepackage[top=2in, bottom=1.5in, left=1.1in, right=1.1in]{geometry}

\newcommand{\sr}[3]{\underset{\text{#2}}{\mbox{#1}}}

\makeatletter
\renewcommand{\@makefnmark}{%
  \hbox{\kern0.15em\@textsuperscript{\normalfont\@thefnmark}}%
}
\makeatother

\newtheorem{Def}{Definition}[section]
\newtheorem{Th}[Def]{Theorem}
\newtheorem{Ex}[Def]{Example}
\newtheorem{Lemma}[Def]{Lemma}
\newtheorem{Prop}[Def]{Proposition}
\newtheorem{Cor}[Def]{Corollary}
\newtheorem{Rem}[Def]{Remark}

{}
\newtheorem{Prob}[Def]{Problem}
\newtheorem{Conj}[Def]{Conjecture}

\newtheorem{Fact}[Def]{Fact}

\DeclareMathOperator{\S1k}{S_1^\kappa}
\DeclareMathOperator{\Sk}{S_{< \kappa}^\kappa}
\DeclareMathOperator{\Uk}{U_{< \kappa}^\kappa}

\DeclareMathOperator{\bk}{\mathfrak{b}_\kappa}
\DeclareMathOperator{\dk}{\mathfrak{d}_\kappa}
\DeclareMathOperator{\pk}{\mathfrak{p}_\kappa}
\DeclareMathOperator{\dlk}{\mathfrak{d}_\kappa^{<\kappa}}
\DeclareMathOperator{\fink}{Q_\kappa}
\DeclareMathOperator{\maxk}{\max\nolimits_{<\kappa}}
\DeclareMathOperator{\lessk}{\mathord{<}\kappa}
\DeclareMathOperator{\ink}{in(\kappa)}

\DeclareMathOperator{\cf}{cf}
\DeclareMathOperator{\non}{non}
\DeclareMathOperator{\cov}{cov}
\DeclareMathOperator{\dom}{dom}

\DeclareMathAlphabet{\mybb}{U}{BOONDOX-ds}{m}{n}

\begin{document}
\title[Selection principles in an uncountable setting]{Selection principles and the Scheepers diagram in an uncountable setting}
\author[ N. Alam, D. Chandra ]{ Nur Alam$^\dag$, Debraj Chandra$^*$ }
\newcommand{\acr}{\newline\indent}
\address{\llap{\dag\,}Department of Mathematics, Mirza Ghalib College, Gaya, Bihar-823001, India}
\email{nurrejwana@gmail.com}
\address{\llap{*\,}Department of Mathematics, University of Gour Banga, Malda-732103, West Bengal, India}
\email{debrajchandra1986@gmail.com}


\maketitle

\begin{abstract}
We continue the study of selection principles in the uncountable setting initiated in \cite{kw20} and further developed in \cite{ucsp}. We investigate several problems posed in \cite{ucsp} concerning the Scheepers diagram, the $\kappa$-Hurewicz Conjecture, and related questions on $\kappa$-Menger and $K_\kappa$-spaces. We determine the Scheepers diagram in the uncountable setting up to three unresolved implications and compute the critical cardinalities of several properties occurring in the diagram. For weakly compact $\kappa$, we show that there is a model of ZFC in which the $\kappa$-Hurewicz Conjecture fails; in particular, it cannot hold under GCH at $\kappa$. Moreover, any model in which $\kappa$ is weakly compact and the $\kappa$-Hurewicz Conjecture holds must satisfy $2^{\mathfrak{b}_\kappa} = 2^\kappa$. We also study the existence of $\bk$-scale sets which are $K_\kappa$-spaces and derive a necessary tree-theoretic condition for such an example. Under the additional assumption of the Perfect Subtree Property, no $\mathfrak{b}_\kappa$-scale set is a $K_\kappa$-space. Consequently, under this assumption, every $\mathfrak{b}_\kappa$-scale set yields an alternative example of a $\kappa$-Menger space which is not a $K_\kappa$-space.
\end{abstract}

\smallskip

\noindent{\bf\keywordsname{}:} {Selection principles, uncountable setting, Scheepers diagram, critical cardinalities, $\kappa$-Menger, $\kappa$-Scheepers, $\kappa$-Hurewicz, $\kappa$-Rothberger, $K_\kappa$-spaces, $\mathfrak{b}_\kappa$-scales.}

\section{Definitions and Terminology}
\label{sec-dt}
Throughout the paper, $\kappa$ denotes an uncountable cardinal, and by a space we mean a topological space, unless otherwise stated. For undefined notions and terminology, we refer the reader to \cite{ucsp, kw20, Engelking, jech03}.

\subsection{Generalized notions of covers and selection principles}
\label{sec-csp}
Recall from \cite{ucsp} (see also \cite{kw20}) the following generalized notions of covers and selection principles. Let $\mathcal{U}$ be a cover of a space $X$. The family $\mathcal{U}$ is called a $\kappa$-cover of $X$ if $X\notin \mathcal{U}$ and for every $F\in [X]^{<\kappa}$, there exists $U\in \mathcal{U}$ such that $F\subseteq U$. The family $\mathcal{U}$ is called a $\gamma_\kappa$-cover of $X$ if it has size $\kappa$ and for each $x\in X$, the set $\{U\in \mathcal{U} : x\notin U\}$ has size less than $\kappa$. We use $\mathcal{O}_\kappa$, $\Omega_\kappa$, and $\Gamma_\kappa$ to denote, respectively, the classes of all open covers of size $\kappa$, all open $\kappa$-covers of size $\kappa$, and all open $\gamma_\kappa$-covers of spaces. If $\mathcal{A}$ is a class of covers of spaces, then $\mathcal{A}(X)$ denotes the family of all covers of $X$ belonging to $\mathcal{A}$. Let $\mathcal{A}, \mathcal{B}\in \{\mathcal{O}_\kappa, \Omega_\kappa, \Gamma_\kappa\}$. With this notation, the generalized selection principles for a space $X$ are defined as follows.

\begin{enumerate}[leftmargin=2cm,itemsep=0.3em]
\item[$\S1k(\mathcal{A},\mathcal{B})$:] For each $(\mathcal{U}_\alpha : \alpha< \kappa) \in \mathcal{A}(X)^\kappa$, there are $U_\alpha\in \mathcal{U}_\alpha$ such that $\{U_\alpha : \alpha< \kappa\}\in \mathcal{B}(X)$.

\item[$\Sk(\mathcal{A},\mathcal{B})$:] For each $(\mathcal{U}_\alpha : \alpha< \kappa) \in \mathcal{A}(X)^\kappa$, there are $\mathcal{V}_\alpha\in [\mathcal{U}_\alpha]^{<\kappa}$ such that $\bigcup_{\alpha< \kappa} \mathcal{V}_\alpha \in \mathcal{B}(X)$.

\item[$\Uk(\mathcal{A},\mathcal{B})$:] For each $(\mathcal{U}_\alpha : \alpha< \kappa) \in \mathcal{A}(X)^\kappa$ with no $\mathcal{U}_\alpha$ has a subcover of size less than $\kappa$, there are $\mathcal{V}_\alpha\in [\mathcal{U}_\alpha]^{<\kappa}$ such that $\{\bigcup \mathcal{V}_\alpha : \alpha< \kappa\}\in \mathcal{B}(X)$.
\end{enumerate}

\begin{Fact}[{\cite[Fact~2.4]{ucsp}}]
\label{fact1}
Let $\mathcal{U}$ be a $\kappa$-cover of a space $X$. Then for every $A\in [X]^{<\kappa}$, there are $\kappa$ many sets $U\in \mathcal{U}$ such that $A\subseteq U$.
\end{Fact}

\begin{Lemma}
\label{lemincl}
For a regular cardinal $\kappa$, every $\gamma_\kappa$-cover of a space is a $\kappa$-cover of size $\kappa$.
\end{Lemma}

Hence for a regular cardinal $\kappa$, the inclusions $\Gamma_\kappa(X) \subseteq \Omega_\kappa(X) \subseteq \mathcal{O}_\kappa(X)$ hold.

\begin{Rem}
\label{rem9}
Let $\kappa$ be a regular cardinal and let $X$ be a space with $|X|< \kappa$. Then $\Omega_\kappa(X) = \Gamma_\kappa(X) = \emptyset$. Moreover, if $\kappa$ is strongly inaccessible, then $\mathcal{O}_\kappa(X) = \emptyset$. Consequently, when $\kappa$ is strongly inaccessible, all the selection principles considered here are satisfied vacuously by every space $X$ with $|X|< \kappa$. Therefore, when considering nontrivial instances of these selection principles, we restrict our attention to spaces of cardinality at least $\kappa$.
\end{Rem}

Recall from \cite{ucsp} (see also \cite{kw20}) that a space $X$ is said to be $\kappa$-Hurewicz (respectively, $\kappa$-Menger and $\kappa$-Rothberger) if it satisfies $\Uk(\mathcal{O}_\kappa, \Gamma_\kappa)$ (respectively, $\Sk(\mathcal{O}_\kappa, \mathcal{O}_\kappa)$ and $\S1k(\mathcal{O}_\kappa, \mathcal{O}_\kappa)$). Similarly, a space $X$ is said to be $\kappa$-Scheepers if it satisfies $\Uk(\mathcal{O}_\kappa, \Omega_\kappa)$. A space $X$ is said to be $\kappa$-compact if every open cover of $X$ has a subcover of size less than $\kappa$. A space $X$ is said to be $\kappa$-Lindel\"{o}f if every open cover of $X$ has a subcover of size at most $\kappa$. A space $X$ is said to be $K_\kappa$ if it is union of $\kappa$ many $\kappa$-compact spaces.

\subsection{$\kappa$-Cantor space and $\kappa$-Baire space}
\label{sec-cb}
Let $\kappa$ be a regular cardinal, and let $(X_\alpha : \alpha< \kappa)$ be a family of discrete topological spaces. The $\kappa$-product topology on $\prod_{\alpha< \kappa} X_\alpha$ is the topology generated by the sets $\prod_{\alpha< \kappa} U_\alpha$, where $U_\alpha \subseteq X_\alpha$ for each $\alpha< \kappa$ and $U_\alpha = X_\alpha$ for all but less than $\kappa$ many $\alpha< \kappa$; see \cite{ucsp}. In this setting, the $\kappa$-product topology is precisely the $\kappa$-box topology; see, for example, \cite{cg16}.

The space $2^\kappa$, endowed with the $\kappa$-product topology, is called the $\kappa$-Cantor space. Identifying each subset of $\kappa$ with its characteristic function in $2^\kappa$, we equip $P(\kappa)$ with the topology induced by this identification. Similarly, the space $\kappa^\kappa$, endowed with the $\kappa$-product topology, is called the $\kappa$-Baire space. We identify each $a\in [\kappa]^\kappa$ with the increasing enumeration of its elements, regarded as an element of $\kappa^\kappa$. Under these identifications, the topologies on $[\kappa]^\kappa$ inherited from $2^\kappa$ and $\kappa^\kappa$ coincide.

A space $X$ is said to be a $P_\kappa$-space \cite{ucsp} if intersection of less than $\kappa$-many open subsets of $X$ is open. Both the $\kappa$-Cantor space and the $\kappa$-Baire space are zero-dimensional Tychonoff $P_\kappa$-spaces. Moreover, if $\kappa^{<\kappa} = \kappa$, then both spaces have bases of size $\kappa$. A cardinal $\kappa$ is weakly compact if $2^\kappa$ is $\kappa$-compact \cite[Theorem~8.23]{cn74}. A weakly compact cardinal $\kappa$ is strongly inaccessible; that is, $\kappa$ is regular and $2^\alpha< \kappa$ for every $\alpha< \kappa$. In particular, $\kappa^{<\kappa} = \kappa$.

\subsection{Cardinal characteristics and $\bk$-scales}
\label{sec-car}
Identifying the members of $[\kappa]^\kappa$ with their increasing enumerations, for $a,b\in [\kappa]^\kappa$, we write $a\leq_\kappa^* b$ if the set $\{\alpha< \kappa : a(\alpha)> b(\alpha)\}$ has size less than $\kappa$, and write $a\leq_\kappa^\kappa b$ if the set $\{\alpha< \kappa : a(\alpha)\leq b(\alpha)\}$ has size $\kappa$. Thus $a\not \leq_\kappa^* b$ if and only if $b\leq_\kappa^\kappa a$.

A family $A\subseteq [\kappa]^\kappa$ is bounded if there exists $b\in [\kappa]^\kappa$ such that $a\leq_\kappa^* b$ for every $a\in A$. The generalized bounding number $\bk$ is the least cardinality of an unbounded subset of $[\kappa]^\kappa$. A family $A\subseteq [\kappa]^\kappa$ is dominating if for every $x\in [\kappa]^\kappa$, there exists $a\in A$ such that $x\leq_\kappa^* a$. The generalized dominating number $\dk$ is the least cardinality of such a dominating family. Equivalently, the notions of boundedness and domination, as well as the cardinals $\bk$ and $\dk$, may be defined using $\kappa^\kappa$ in place of $[\kappa]^\kappa$. The systematic study of the cardinal characteristics $\bk$ and $\dk$ was initiated by Cummings and Shelah \cite{cs95}.

For $a, b\in [\kappa]^\kappa$, we write $a\subseteq_\kappa^* b$ if $|a\setminus b|<\kappa$. A family $A\subseteq [\kappa]^\kappa$ is called $\kappa$-centered if $|\bigcap B| = \kappa$ for every nonempty $B\in [A]^{<\kappa}$. A $\kappa$-pseudointersection of $A\subseteq [\kappa]^\kappa$ is an element $b\in [\kappa]^\kappa$ such that $b\subseteq_\kappa^* a$ for every $a\in A$. The cardinal $\pk$ is defined as the least cardinality of a $\kappa$-centered family $A\subseteq [\kappa]^\kappa$ having no $\kappa$-pseudointersection.

A family $X = \{x_\alpha : \alpha< \mathfrak{b}_\kappa\} \subseteq [\kappa]^\kappa$ is called a $\mathfrak{b}_\kappa$-scale if it is unbounded and $x_\alpha \leq_\kappa^* x_\beta$ for all $\alpha< \beta< \mathfrak{b}_\kappa$. Such scales exist in ZFC. Moreover, $\mathfrak{b}_\kappa$ is regular and $\kappa< \mathfrak{b}_\kappa$. Let $\fink = [\kappa]^{<\kappa}$. A set of the form $X\cup \fink$, where $X$ is a $\mathfrak{b}_\kappa$-scale, is called a $\mathfrak{b}_\kappa$-scale set.

For a topological property $Q$, the critical cardinality of $Q$, denoted by $\non(Q)$, is defined by \[\non(Q) = \min \{|X| : X\subseteq P(\kappa) \text{ and } X\text{ does not satisfy } Q\}.\]

\section{Background and Motivation}
\label{sec-bm}
In \cite{coc1} (see also \cite{coc2}), Scheepers initiated the systematic
study of selection principles in general and set-theoretic topology in the classical countable setting. The investigation of their uncountable
analogues was initiated by Korch and Weiss in \cite{kw20} and was further
developed in \cite{ucsp}. In the countable setting, Menger's Conjecture,
the Hurewicz Conjecture, and the Hurewicz Problem, together with their
solutions and related developments, have played an important role in the
theory of selection principles and have been studied extensively; see, for
example, \cite{bs01, bt06, fm88, ot11, tsaban11, tsaban08}.

Motivated by these classical problems, in \cite{ucsp}, the authors considered their uncountable analogues, namely, the $\kappa$-Menger Conjecture, the $\kappa$-Hurewicz Conjecture, and the $\kappa$-Hurewicz Problem. For a weakly compact cardinal $\kappa$, the $\kappa$-Menger Conjecture asserts that a subspace of $P(\kappa)$ is a $K_\kappa$-space if and only if it is $\kappa$-Menger, whereas the $\kappa$-Hurewicz Conjecture asserts that a subspace of $P(\kappa)$ is a $K_\kappa$-space if and only if it is $\kappa$-Hurewicz. The $\kappa$-Hurewicz Problem asks whether there exists a $\kappa$-Menger subspace of $P(\kappa)$ which is not $\kappa$-Hurewicz. Assuming that $\kappa$ is weakly compact, the authors refuted the $\kappa$-Menger Conjecture and answered the $\kappa$-Hurewicz Problem affirmatively. The $\kappa$-Hurewicz Conjecture, however, was left unresolved, leading to the following question: Is the $\kappa$-Hurewicz Conjecture consistent with ZFC + $\kappa$ is weakly compact? (see \cite[Problem~8.1]{ucsp})

In connection, the authors also posed the following questions:
\begin{enumerate}
\item Is it consistent with ZFC + $\kappa$ is weakly compact that there exists a $\bk$-scale set which is a $K_\kappa$-space? (see \cite[Problem~8.2]{ucsp})

\item Is there an alternative construction of a $\kappa$-Menger space which is not a $K_\kappa$-space, different from the construction used in the solution of the $\kappa$-Hurewicz Problem? (see \cite[Problem~8.3]{ucsp})
\end{enumerate}

Another central direction in the classical theory of selection principles is the study of the Scheepers diagram and the critical cardinalities of the selection properties occurring in it; see, for example, \cite{coc2, tsaban24, peng23}. In the uncountable setting, a few implications among the corresponding generalized selection principles were established in \cite{ucsp}, where the critical cardinalities of some of these properties were also determined. In view of the classical Scheepers diagram, they asked how the Scheepers diagram behaves in the uncountable setting and whether the remaining properties can be separated, possibly consistently (see \cite[Problem~8.4]{ucsp}). Thus understanding the implication structure and possible separations among uncountable selection principles, together with their critical cardinalities, forms a natural continuation of the study initiated in \cite{kw20, ucsp}.

In this paper, we continue the study of selection principles in the uncountable setting initiated in \cite{kw20,ucsp}, with particular emphasis on Problems~8.1--8.4 of \cite{ucsp}. In response to \cite[Problem~8.4]{ucsp}, we determine the Scheepers diagram in the uncountable setting up to three unresolved implications (see Figure~\ref{sdu}). More precisely, it remains open whether any of the implications $\Uk(\mathcal{O}_\kappa, \Gamma_\kappa) \rightarrow \Sk(\Gamma_\kappa, \Omega_\kappa)$, $\Uk(\mathcal{O}_\kappa, \Omega_\kappa) \rightarrow \Sk(\Gamma_\kappa, \Omega_\kappa)$, and $\Sk(\Gamma_\kappa, \Omega_\kappa) \rightarrow \Sk(\Omega_\kappa, \Omega_\kappa)$ holds. In connection with this analysis, we determine the critical cardinalities of $\Sk(\Gamma_\kappa, \Omega_\kappa)$, $\Uk(\mathcal{O}_\kappa, \Omega_\kappa)$, $\S1k(\Gamma_\kappa, \Omega_\kappa)$, $\S1k(\Gamma_\kappa, \mathcal{O}_\kappa)$, and $\S1k(\mathcal{O}_\kappa, \mathcal{O}_\kappa)$. The critical cardinalities of $\Sk(\Omega_\kappa, \Omega_\kappa)$ and $\S1k(\Omega_\kappa, \Omega_\kappa)$ remain open.

In Section~\ref{sec-hc}, we investigate \cite[Problem~8.1]{ucsp}. We show that there is a model of ZFC in which $\kappa$ is weakly compact and the $\kappa$-Hurewicz Conjecture fails. In particular, the $\kappa$-Hurewicz Conjecture cannot hold under GCH at $\kappa$, that is, when $2^\kappa=\kappa^+$. Although this does not settle \cite[Problem~8.1]{ucsp}, it yields a necessary condition for a positive answer: any model in which $\kappa$ is weakly compact and the $\kappa$-Hurewicz Conjecture holds must satisfy $2^{\bk} = 2^\kappa$.

Finally, in Section~\ref{sec-fp}, we consider \cite[Problems~8.2 and~8.3]{ucsp}. We derive a necessary tree-theoretic condition for a positive answer to \cite[Problem~8.2]{ucsp}. As a consequence, under the additional assumption of the Perfect Subtree Property \cite{hm23}, we obtain a negative answer to \cite[Problem~8.2]{ucsp}; the question under weak compactness alone remains open. Moreover, under the Perfect Subtree Property, every $\bk$-scale set yields an alternative example of a $\kappa$-Menger space which is not a $K_\kappa$-space, and hence gives a positive answer to \cite[Problem~8.3]{ucsp} under this additional assumption.

\section{The Scheepers diagram in the uncountable setting}
\label{sec-sd}
In this section, we consider the following problem posed in \cite{ucsp}.

\begin{Prob}[{\cite[Problem~8.4]{ucsp}}]
\label{prob9}
What is the appearance of the Scheepers diagram in the uncountable case? How about the separations of other properties (even consistently)?
\end{Prob}

\subsection{Implications and equivalences}
\label{sec-ie}
We do not consider the selection principles $\Sk(\mathcal{O}_\kappa, \Omega_\kappa)$, $\Sk(\mathcal{O}_\kappa, \Gamma_\kappa)$, $\S1k(\mathcal{O}_\kappa, \Omega_\kappa)$, and $\S1k(\mathcal{O}_\kappa, \Gamma_\kappa)$ separately, since, when $\kappa$ is strongly inaccessible and the underlying space is $T_1$, these principles do not give rise to nontrivial covering properties.

We note that some of the equivalences and implications established in this section are proved by arguments analogous to those used for their classical countable counterparts in \cite{coc1,coc2}.

\begin{Th}
\label{thm3}
Assume that $\kappa$ is regular. For a space $X$, the following properties are equivalent.
\begin{multicols}{2}
\begin{enumerate}[wide=0pt,label={\upshape(\arabic*)}]
  \item $\Sk(\mathcal{O}_\kappa, \mathcal{O}_\kappa)$
  \item $\Sk(\Omega_\kappa,\mathcal{O}_\kappa)$
  \item $\Sk(\Gamma_\kappa,\mathcal{O}_\kappa)$
  \item $\Uk(\mathcal{O}_\kappa,\mathcal{O}_\kappa)$
  \item $\Uk(\Omega_\kappa,\mathcal{O}_\kappa)$
  \item $\Uk(\Gamma_\kappa,\mathcal{O}_\kappa)$.
\end{enumerate}
\end{multicols}
\end{Th}

\begin{proof}
First observe that for $\mathcal{A} \in \{\mathcal{O}_\kappa, \Omega_\kappa, \Gamma_\kappa\}$, the properties $\Sk(\mathcal{A}, \mathcal{O}_\kappa)$ and
$\Uk(\mathcal{A}, \mathcal{O}_\kappa)$ are equivalent. The implication $\Sk(\mathcal{A}, \mathcal{O}_\kappa)\rightarrow \Uk(\mathcal{A}, \mathcal{O}_\kappa)$ is immediate. Conversely, suppose that $X$ satisfies $\Uk(\mathcal{A}, \mathcal{O}_\kappa)$, and let $(\mathcal{U}_\alpha : \alpha< \kappa) \in \mathcal{A}(X)^\kappa$. We assume that no $\mathcal{U}_\alpha$ has a subcover of size less than $\kappa$. By $\Uk(\mathcal{A}, \mathcal{O}_\kappa)$, there are $\mathcal{V}_\alpha \in [\mathcal{U}_\alpha]^{<\kappa}$, $\alpha< \kappa$, such that $\{\bigcup \mathcal{V}_\alpha : \alpha< \kappa\} \in \mathcal{O}_\kappa(X)$. In particular, $\bigcup_{\alpha< \kappa} \mathcal{V}_\alpha$ covers $X$. If necessary, enlarge each $\mathcal{V}_\alpha$ by one member of $\mathcal{U}_\alpha$, chosen recursively so that these additional members are pairwise distinct. The resulting selections still have size less than $\kappa$, while their union is an open cover of $X$ of size $\kappa$. Therefore, $X$ satisfies $\Sk(\mathcal{A}, \mathcal{O}_\kappa)$.

Similarly, one can prove that $\Sk(\mathcal{O}_\kappa, \mathcal{O}_\kappa)$, $\Sk(\Omega_\kappa, \mathcal{O}_\kappa)$, and $\Sk(\Gamma_\kappa, \mathcal{O}_\kappa)$ are equivalent.
\end{proof}

\begin{Th}
\label{thm4}
Assume that $\kappa$ is regular. For a space $X$, the following properties are equivalent.
\begin{enumerate}[wide=0pt,label={\upshape(\arabic*)}]
  \item $\Uk(\mathcal{O}_\kappa, \Omega_\kappa)$
  \item $\Uk(\Omega_\kappa, \Omega_\kappa)$.
  \item $\Uk(\Gamma_\kappa, \Omega_\kappa)$.
\end{enumerate}
\end{Th}

\begin{proof}
We only sketch the proof of $(3)\Rightarrow (1)$, as the other implications are immediate. Assume that $X$ satisfies $\Uk(\Gamma_\kappa, \Omega_\kappa)$, and let $(\mathcal{U}_\alpha : \alpha< \kappa) \in \mathcal{O}_\kappa(X)^\kappa$ be such that no $\mathcal{U}_\alpha$ has a subcover of size less than $\kappa$. For each $\alpha< \kappa$, write $\mathcal{U}_\alpha = \{U_\xi^{(\alpha)} : \xi< \kappa\}$. For each $\xi< \kappa$, put $W_\xi^{(\alpha)} = \bigcup_{\eta\leq \xi} U_\eta^{(\alpha)}$, and let $\mathcal{W}_\alpha = \{W_\xi^{(\alpha)} : \xi< \kappa\}$. Clearly, $\mathcal{W}_\alpha \in \Gamma_\kappa(X)$. Then there are $\mathcal{H}_\alpha \in [\mathcal{W}_\alpha]^{< \kappa}$, $\alpha< \kappa$, such that $\{\bigcup \mathcal{H}_\alpha : \alpha< \kappa\}\in \Omega_\kappa(X)$. For each $\alpha< \kappa$, define \[\mathcal{V}_\alpha = \bigcup \{\{U_{\eta}^{(\alpha)} : \eta\leq \xi\}: W_{\xi}^{(\alpha)}\in \mathcal{H}_\alpha\},\] and these selections witness that $X$ satisfies $\Uk(\mathcal{O}_\kappa, \Omega_\kappa)$.
\end{proof}

\begin{Th}
\label{thm5}
Assume that $\kappa$ is regular. For a space $X$, the following properties are equivalent.
\begin{enumerate}[wide=0pt, label={\upshape(\arabic*)}, leftmargin=*]
  \item $\Uk(\mathcal{O}_\kappa, \Gamma_\kappa)$
  \item $\Uk(\Omega_\kappa, \Gamma_\kappa)$
  \item $\Uk(\Gamma_\kappa, \Gamma_\kappa)$.
\end{enumerate}
\end{Th}

\begin{Th}
\label{thm6}
Assume $\kappa^{<\kappa} = \kappa$. For a space $X$, the properties $\S1k(\mathcal{O}_\kappa, \mathcal{O}_\kappa)$ and $\S1k(\Omega_\kappa, \mathcal{O}_\kappa)$ are equivalent.
\end{Th}

\begin{proof}
Obviously, $\S1k(\mathcal{O}_\kappa, \mathcal{O}_\kappa)$ implies $\S1k(\Omega_\kappa, \mathcal{O}_\kappa)$. For the converse, suppose that $X$ satisfies $\S1k(\Omega_\kappa, \mathcal{O}_\kappa)$, and let $(\mathcal{U}_\xi : \xi< \kappa)\in \mathcal{O}_\kappa(X)^\kappa$. Let $\{I_\alpha : \alpha< \kappa\}$ be a partition of $\kappa$ such that $|I_\alpha| = \kappa$ for every $\alpha< \kappa$. Choose $i_\alpha \in I_\alpha$ and $M_\alpha \in \mathcal{U}_{i_\alpha}$ so that the $M_\alpha$'s are pairwise distinct and proper subsets of $X$. This is possible since every $\mathcal{U}_{i_\alpha}$ has size $\kappa$. For each $\alpha< \kappa$, let $\mathcal{V}_\alpha$ consist of all proper open sets of the form $M_\alpha \cup \bigcup_{\xi\in F} U_\xi$, where $F\in [I_\alpha \setminus \{i_\alpha\}]^{<\kappa}$ and $U_\xi \in \mathcal{U}_\xi$ for each $\xi\in F$.

Since $\kappa^{<\kappa} = \kappa$, we have $|\mathcal{V}_\alpha| \leq \kappa$. Suppose that there exist $\alpha< \kappa$ such that $\mathcal{V}_\alpha$ is not a $\kappa$-cover of $X$. Then there exists $A\in [X]^{<\kappa}$ such that $A$ is contained in no member of $\mathcal{V}_\alpha$. Choose pairwise distinct
$\xi_x\in I_\alpha \setminus \{i_\alpha\}$, $x\in A$, and choose $U_{\xi_x} \in \mathcal{U}_{\xi_x}$ such that $x\in U_{\xi_x}$. Then $A\subseteq
M_\alpha \cup \bigcup_{x\in A} U_{\xi_x}$. By the choice of $A$, the set on the right cannot be a proper subset of $X$. Hence $X = M_\alpha \cup \bigcup_{x\in A} U_{\xi_x}$.

Selecting these sets from the corresponding covers and choosing one member from each of the remaining covers, we obtain a selection from $(\mathcal{U}_\xi : \xi< \kappa)$ that covers $X$. The remaining members can be chosen to be pairwise distinct and distinct from those already selected, since each $\mathcal{U}_\xi$ has size $\kappa$. As there are $\kappa$ many remaining indices, the selected family has size $\kappa$, and therefore belongs to $\mathcal{O}_\kappa(X)$.

We may therefore assume that each $\mathcal{V}_\alpha \in \Omega_\kappa(X)$. Applying $\S1k(\Omega_\kappa, \mathcal{O}_\kappa)$, choose
$V_\alpha \in \mathcal{V}_\alpha$, $\alpha< \kappa$, such that $\{V_\alpha : \alpha< \kappa\} \in \mathcal{O}_\kappa(X)$. For every $\alpha< \kappa$, write $V_\alpha = M_\alpha \cup \bigcup_{\xi\in F_\alpha} U_\xi^{(\alpha)}$, where $F_\alpha \in [I_\alpha \setminus \{i_\alpha\}]^{<\kappa}$ and $U_\xi^{(\alpha)} \in \mathcal{U}_\xi$. From $\mathcal{U}_{i_\alpha}$ select $M_\alpha$, and from $\mathcal{U}_\xi$, $\xi\in F_\alpha$, select $U_\xi^{(\alpha)}$. Choose arbitrary members from all remaining $\mathcal{U}_\xi$'s. Thus $X$ satisfies $\S1k(\mathcal{O}_\kappa, \mathcal{O}_\kappa)$.
\end{proof}

\begin{Th}
\label{thm7}
Assume that $\kappa$ is regular. For a $P_\kappa$-space, the properties $\S1k(\Gamma_\kappa, \Gamma_\kappa)$ and $\Sk(\Gamma_\kappa, \Gamma_\kappa)$ are equivalent.
\end{Th}

\begin{proof}
Clearly, $\S1k(\Gamma_\kappa, \Gamma_\kappa)$ implies $\Sk(\Gamma_\kappa, \Gamma_\kappa)$. We now prove the converse. Suppose that $X$ satisfies $\Sk(\Gamma_\kappa, \Gamma_\kappa)$, and let $(\mathcal{U}_\xi : \xi< \kappa)\in \Gamma_\kappa(X)^\kappa$. Removing $X$ from each $\mathcal{U}_\xi$, if necessary, we may assume that $X\notin \mathcal{U}_\xi$. Write $\mathcal{U}_\xi = \{U_\beta^{(\xi)} : \beta< \kappa\}$, $\xi< \kappa$, where the members of each enumeration are pairwise distinct. For $\alpha, \beta< \kappa$, put $W_\beta^{(\alpha)} = \bigcap_{\xi\leq \alpha} U_\beta^{(\xi)}$, and let $\mathcal{W}_\alpha = \{W_\beta^{(\alpha)} : \beta< \kappa\}$. Then $\mathcal{W}_\alpha \in \Gamma_\kappa(X)$ for every $\alpha< \kappa$.

Apply $\Sk(\Gamma_\kappa,\Gamma_\kappa)$ to $(\mathcal{W}_\alpha : \alpha< \kappa)$. Choose $\mathcal{V}_\alpha \in [\mathcal{W}_\alpha]^{<\kappa}$, $\alpha< \kappa$, such that $\mathcal{V} = \bigcup_{\alpha< \kappa} \mathcal{V}_\alpha \in \Gamma_\kappa(X)$. Select a set $A\in [\kappa]^\kappa$ and pairwise distinct members $W_\alpha \in \mathcal{V}_\alpha$, $\alpha\in A$. Then $\{W_\alpha : \alpha \in A\}$ is a $\gamma_\kappa$-cover of $X$. For each $\alpha\in A$, choose $\beta_\alpha < \kappa$ such that $W_\alpha = W_{\beta_\alpha}^{(\alpha)} = \bigcap_{\xi\leq \alpha} U_{\beta_\alpha}^{(\xi)}$. Since $\kappa$ is regular and $|A| = \kappa$, the set $A$ is cofinal in $\kappa$. For each $\xi< \kappa$, let $d(\xi) = \min (A\setminus \xi)$, and select $U_\xi = U_{\beta_{d(\xi)}}^{(\xi)} \in \mathcal{U}_\xi$.

We show that $\{U_\xi : \xi< \kappa\} \in \Gamma_\kappa(X)$. Fix $x\in X$. If $x\notin U_\xi$, then since $\xi \leq d(\xi)$, $x\notin W_{d(\xi)}$. Therefore, \[\{\xi< \kappa : x\notin U_\xi\} \subseteq \bigcup_{\substack{\alpha\in A\\ x\notin W_\alpha}} d^{-1}(\{\alpha\}).\] For every $\alpha\in A$, $|d^{-1}(\{\alpha\})| < \kappa$, and since $\{W_\alpha : \alpha\in A\}$ is a $\gamma_\kappa$-cover, the set $\{\alpha\in A : x\notin W_\alpha\}$ has size less than $\kappa$. By regularity of $\kappa$, the set $\{\xi< \kappa : x\notin U_\xi\}$ also has size less than $\kappa$.

Finally, each $U_\xi$ is a proper subset of $X$. Hence no fixed $U_\xi$ can occur for $\kappa$ many indices; otherwise, choosing a point outside that set would contradict the preceding conclusion. Since $\kappa$ is regular, it follows that the size of $\{U_\xi : \xi< \kappa\}$ is $\kappa$. Thus $\{U_\xi : \xi< \kappa\} \in \Gamma_\kappa(X)$. Therefore, $X$ satisfies $\S1k(\Gamma_\kappa, \Gamma_\kappa)$.
\end{proof}

\begin{Th}
\label{thm8}
Assume $\kappa^{<\kappa} = \kappa$. For a $P_\kappa$-space $X$, the properties $\S1k(\Omega_\kappa, \Gamma_\kappa)$ and $\Sk(\Omega_\kappa, \Gamma_\kappa)$ are equivalent.
\end{Th}

\begin{proof}
We only prove that $\Sk(\Omega_\kappa, \Gamma_\kappa)$ implies $\S1k(\Omega_\kappa, \Gamma_\kappa)$, as the converse is immediate. Suppose that $X$ satisfies $\Sk(\Omega_\kappa, \Gamma_\kappa)$, and let $(\mathcal{U}_\xi : \xi< \kappa) \in \Omega_\kappa(X)^\kappa$. For each $\alpha< \kappa$, define $\mathcal{W}_\alpha = \{\bigcap_{\xi\leq \alpha} U_\xi : U_\xi \in \mathcal{U}_\xi \text{ for each } \xi\leq \alpha\}$. Then $\mathcal{W}_\alpha \in \Omega_\kappa(X)$ for every $\alpha< \kappa$. Indeed, if $F\in [X]^{<\kappa}$, then for every $\xi\leq \alpha$, there is $U_\xi\in \mathcal{U}_\xi$ such that $F\subseteq U_\xi$. Hence $F\subseteq \bigcap_{\xi\leq \alpha} U_\xi$. Also no member of $\mathcal{W}_\alpha$ is equal to $X$, since every $U_\xi\in \mathcal{U}_\xi$ is a proper subset of $X$. Furthermore, \[|\mathcal{W}_\alpha| \leq \kappa^{|\alpha+1|} \leq \kappa^{<\kappa} = \kappa,\] and therefore, $|\mathcal{W}_\alpha| = \kappa$.

The remainder of the proof proceeds as in the proof of Theorem~\ref{thm3}.
\end{proof}

\begin{Prop}
Assume that $\kappa$ is regular. For a space $X$, $\Sk(\Gamma_\kappa, \Omega_\kappa)$ implies $\Uk(\mathcal{O}_\kappa, \Omega_\kappa)$.
\end{Prop}

\begin{Prop}
For a space $X$, $\Uk(\mathcal{O}_\kappa, \Omega_\kappa)$ implies $\Sk(\mathcal{O}_\kappa, \mathcal{O}_\kappa)$.
\end{Prop}

\begin{Rem}
\label{rem7}
For a space $X$, the following implications are immediate.
\begin{enumerate}[wide=0pt, label={\upshape(\arabic*)}, leftmargin=*]
  \item $\S1k(\Gamma_\kappa, \Omega_\kappa) \rightarrow \Sk(\Gamma_\kappa, \Omega_\kappa)$; $\S1k(\Gamma_\kappa, \Omega_\kappa) \rightarrow \S1k(\Gamma_\kappa, \mathcal{O}_\kappa)$; $\S1k(\mathcal{O}_\kappa, \mathcal{O}_\kappa) \rightarrow \S1k(\Gamma_\kappa, \mathcal{O}_\kappa)$; and $\S1k(\Omega_\kappa, \Omega_\kappa) \rightarrow \Sk(\Omega_\kappa, \Omega_\kappa)$.

  \item If $\kappa$ is regular, then $\S1k(\Omega_\kappa, \Gamma_\kappa) \rightarrow \S1k(\Gamma_\kappa, \Gamma_\kappa) \rightarrow \S1k(\Gamma_\kappa, \Omega_\kappa)$; $\S1k(\Omega_\kappa, \Gamma_\kappa) \rightarrow \S1k(\Omega_\kappa, \Omega_\kappa)  \rightarrow \S1k(\Gamma_\kappa, \Omega_\kappa)$; $\Sk(\Omega_\kappa, \Omega_\kappa) \rightarrow \Sk(\Gamma_\kappa, \Omega_\kappa)$; $\S1k(\Gamma_\kappa, \mathcal{O}_\kappa) \rightarrow \Sk(\mathcal{O}_\kappa, \mathcal{O}_\kappa)$; and $\S1k(\Gamma_\kappa, \Gamma_\kappa) \rightarrow \Uk(\mathcal{O}_\kappa, \Gamma_\kappa) \rightarrow \Uk(\mathcal{O}_\kappa, \Omega_\kappa)$.

  \item If $\kappa^{<\kappa} = \kappa$, then $\S1k(\Omega_\kappa, \Omega_\kappa)  \rightarrow \S1k(\mathcal{O}_\kappa, \mathcal{O}_\kappa)$.
\end{enumerate}
\end{Rem}

\subsection{Non-implications and critical cardinalities}
\label{sec-ni}
We begin by recalling the following results from \cite{ucsp} concerning the critical cardinalities of certain selection principles.

\begin{Th}[{\cite[Corollary~3.2]{ucsp}}]
\label{thm14}
Assume $\kappa^{<\kappa} = \kappa$. Then $\non(\kappa\text{-Hurewicz}) = \bk$.
\end{Th}

\begin{Th}[{\cite[Lemma~3.3]{ucsp}}]
\label{thm16}
Assume $\kappa^{<\kappa} = \kappa$. Then $\non(\S1k(\Gamma_\kappa, \Gamma_\kappa)) = \bk$.
\end{Th}

\begin{Th}[{\cite[Corollary~4.3]{ucsp}}]
\label{thm15}
Assume $\kappa^{<\kappa} = \kappa$. Then $\non(\kappa\text{-Menger}) = \dk$.
\end{Th}

\begin{Th}[{\cite[Corollary~6.5]{ucsp}}]
\label{thm17}
Assume $\kappa^{<\kappa} = \kappa$. Then $\non(\S1k(\Omega_\kappa, \Gamma_\kappa)) = \pk$.
\end{Th}

For $A\subseteq \kappa^\kappa$, we define \[\maxk(A) = \left\{\alpha \mapsto \sup \{f(\alpha) : f\in F\} : F\in [A]^{<\kappa}\right\}\] with the usual convention that $\sup \emptyset = 0$. Following \cite[Definition~1.11]{mrs05}, a family $A \subseteq \kappa^\kappa$ is said to be $\lessk$-dominating if for each $g\in \kappa^\kappa$, there exists $F\in [A]^{<\kappa}$ such that $g(\alpha) < \sup \{f(\alpha) : f\in F\}$ for all $\alpha< \kappa$. The $\lessk$-dominating number is defined by \[\dlk = \min \{|A| : A\subseteq \kappa^\kappa \text{ is $\lessk$-dominating}\}.\]

\begin{Lemma}
\label{lemma8}
Assume that $\kappa$ is regular. Then \[\dlk = \min \{|A| : A \subseteq \kappa^\kappa \text{ and } \maxk(A) \text{ is dominating}\}.\]
\end{Lemma}

\begin{proof}
Put $\delta = \min \{|A| : A \subseteq \kappa^\kappa \text{ and } \maxk(A) \text{ is dominating}\}$. We first show that $\delta \leq \dlk$. Let $A \subseteq \kappa^\kappa$ be a $\lessk$-dominating family. For each $g\in \kappa^\kappa$, there exists $F\in [A]^{<\kappa}$ such that $g(\alpha) < \sup \{f(\alpha) : f\in F\}$ for all $\alpha < \kappa$. Define $h_F\in \kappa^\kappa$ by $h_F(\alpha) = \sup \{f(\alpha) : f\in F\}$ for all $\alpha< \kappa$. Then $h_F\in \maxk(A)$ and $g\leq_\kappa^* h_F$. Hence $\maxk(A)$ is dominating. It follows that $\delta \leq \dlk$.

For the reverse inequality, let $A \subseteq \kappa^\kappa$ satisfy $|A| = \delta$ and suppose that $\maxk(A)$ is dominating. Notice first that $|A| \geq \kappa$. Indeed, if $|A|< \kappa$, then by the regularity of $\kappa$, the function $h\in \kappa^\kappa$ defined by $h(\alpha) = \sup \{f(\alpha) : f\in A\}$ is well defined. Every member of $\maxk(A)$ is pointwise bounded by $h$, and hence $\maxk(A)$ cannot be dominating, a contradiction.

For $\beta, \gamma< \kappa$, define $p_{\beta,\gamma} \in \kappa^\kappa$ by
\[p_{\beta,\gamma}(\alpha) = \begin{cases}
\gamma, & \mbox{if } \alpha< \beta\\
0, & \mbox{if } \beta \leq \alpha< \kappa.
\end{cases}\]
Let $B = A\cup \{p_{\beta,\gamma} : \beta, \gamma< \kappa\}$. Since $|A| \geq \kappa$, $|B| = |A| = \delta$. We claim that $B$ is $\lessk$-dominating. Let $g\in \kappa^\kappa$, and define $g^+ \in \kappa^\kappa$ by $g^+(\alpha) = g(\alpha) + 1$ for all $\alpha< \kappa$. Since $\maxk(A)$ is dominating, there exists $F\in [A]^{<\kappa}$ such that $g^+ \leq_\kappa^* h_F$, where $h_F(\alpha) = \sup \{f(\alpha) : f\in F\}$. Thus $E = \{\alpha< \kappa : g(\alpha) + 1> h_F(\alpha)\}$ has size less than $\kappa$. Since $\kappa$ is regular, $E$ is not cofinal in $\kappa$. Choose $\beta< \kappa$ such that $E \subseteq \beta$. Again by the regularity of $\kappa$, there exists $\gamma< \kappa$ such that $g(\alpha) < \gamma$ for every $\alpha< \beta$. Set $G = F\cup \{p_{\beta,\gamma}\}$. Clearly, $G\in [B]^{<\kappa}$. If $\alpha< \beta$, then $g(\alpha) < \gamma = p_{\beta,\gamma}(\alpha) \leq \sup \{f(\alpha) : f\in G\}$.

If $\beta \leq \alpha< \kappa$, then $\alpha \notin E$, and hence $g(\alpha) + 1\leq h_F(\alpha)$. Therefore, $g(\alpha) < h_F(\alpha) \leq \sup \{f(\alpha) : f\in G\}$. Consequently, $g(\alpha) < \sup \{f(\alpha) : f\in G\}$ for all $\alpha< \kappa$. Thus $B$ is $\lessk$-dominating, and hence $\dlk \leq |B| = \delta$. Combining the two inequalities, we obtain $\dlk = \delta$.
\end{proof}

The following observations concerning $\dlk$ follow from results of Matet--Ros{\l}anowski--Shelah \cite{mrs05}, together with the definition of $\dlk$. For the convenience of the reader, we include a proof using the preceding lemma.

\begin{Lemma}
\label{lemma9}
Assume that $\kappa$ is regular.
\begin{enumerate}[wide=0pt, label={\upshape(\arabic*)}, leftmargin=*, ref={\theLemma(\arabic*)}]
  \item\label{lemma901} Then $\bk \leq \dlk \leq \dk$.
  \item\label{lemma902} If $\lambda^{<\kappa} < \dk$ for every $\lambda < \dk$, then $\dlk = \dk$.
\end{enumerate}
\end{Lemma}

\begin{proof}
$(1)$. We first prove that $\bk \leq \dlk$. Let $A\subseteq \kappa^\kappa$ with $|A|< \bk$. By the definition of $\bk$, the family $A$ is bounded. Thus there exists $g\in \kappa^\kappa$ such that $f \leq_\kappa^* g$ for every $f\in A$. We claim that $\maxk(A)$ is also bounded by $g$. Indeed, let $h\in \maxk(A)$. Then there exists $F\in [A]^{<\kappa}$ such that $h(\alpha) = \sup \{f(\alpha) : f\in F\}$ for all $\alpha< \kappa$. For each $f\in F$, put $E_f = \{\alpha< \kappa : f(\alpha) > g(\alpha)\}$. Since $f \leq_\kappa^* g$, we have $|E_f|< \kappa$. As $|F|< \kappa$ and $\kappa$ is regular, $|\bigcup_{f\in F} E_f|< \kappa$.

For each $\alpha \in \kappa \setminus \bigcup_{f\in F} E_f$, we have $f(\alpha) \leq g(\alpha)$ for every $f\in F$, and hence $h(\alpha) = \sup \{f(\alpha) : f\in F\} \leq g(\alpha)$. Therefore, $h \leq_\kappa^* g$. Since $h\in \maxk(A)$ was arbitrary, $\maxk(A)$ is bounded and hence cannot be dominating. By Lemma~\ref{lemma8}, $\bk \leq \dlk$.

Since for every dominating family $D \subseteq \kappa^\kappa$, $\maxk(D)$ is also dominating, again by Lemma~\ref{lemma8}, we get $\dlk \leq \dk$.

$(2)$. Suppose that $\lambda^{<\kappa} < \dk$ for every $\lambda< \dk$. Assume, towards a contradiction, that there exists a family $A\subseteq \kappa^\kappa$ such that $\maxk(A)$ is dominating and $|A| = \lambda< \dk$. Since \[|\maxk(A)| \leq |[A]^{<\kappa}| \leq \lambda^{<\kappa},\] we obtain $|\maxk(A)| < \dk$. By the definition of $\mathfrak d_\kappa$, $\maxk(A)$ cannot be dominating, which is a contradiction. Hence by Lemma~\ref{lemma8}, $\dk \leq \dlk$. From $(1)$, we have $\dlk = \dk$.
\end{proof}

The following characterization is the uncountable analogue of the classical characterization of the Scheepers property in terms of continuous images (see \cite[Theorem~2.1]{tsaban01}). Its proof follows the classical argument, with finite maxima replaced by pointwise suprema of fewer than $\kappa$ many functions.

\begin{Lemma}
\label{lemma7}
Assume that $\kappa$ is regular. Let $X$ be a zero-dimensional $\kappa$-Lindel\"{o}f $P_\kappa$-space. Then $X$ is $\kappa$-Scheepers if and only if for every continuous image $Y$ of $X$ into $\kappa^\kappa$, the family $\maxk(Y)$ is not dominating.
\end{Lemma}

\begin{proof}
Suppose first that $X$ is $\kappa$-Scheepers. Let $Y$ be a continuous image of $X$ into $\kappa^\kappa$. We show that $\maxk(Y)$ is not dominating. Suppose, to the contrary, that $\maxk(Y)$ is dominating. Put \[A = \{\alpha< \kappa : \{y(\alpha) : y\in Y\} \text{ is cofinal in } \kappa\}.\] We claim that $|\kappa \setminus A|< \kappa$. Otherwise, for each $\alpha\in \kappa \setminus A$, choose $b_\alpha< \kappa$ such that $y(\alpha)\leq b_\alpha$ for every $y\in Y$. Define $g\in \kappa^\kappa$ by
\[g(\alpha) = \begin{cases}
b_\alpha+1, & \mbox{if } \alpha\in \kappa \setminus A\\
0, & \mbox{otherwise}.
\end{cases}\]
Then for every $F\in [Y]^{<\kappa} \setminus \{\emptyset\}$ and every $\alpha\in \kappa \setminus A$, \[\sup \{f(\alpha) : f\in F\}\leq b_\alpha < b_\alpha+1 = g(\alpha).\] Since $|\kappa \setminus A| = \kappa$, this shows that $g\not \leq_\kappa^* h$ for every $h\in \maxk(Y)$, contradicting the assumption that $\maxk(Y)$ is dominating. Thus $|\kappa \setminus A|< \kappa$, and hence $|A| = \kappa$.

Write $A = \{a_\alpha : \alpha< \kappa\}$. For $\alpha, \xi< \kappa$, let $U_\xi^{(\alpha)} = \{y\in Y : y(a_\alpha) \leq \xi\}$ and $\mathcal{U}_\alpha = \{U_\xi^{(\alpha)} : \xi< \kappa\}$. Since $\{y(a_\alpha) : y\in Y\}$ is cofinal in $\kappa$, $\mathcal{U}_\alpha$ is an open cover of $Y$ of size $\kappa$ having no subcover of cardinality less than $\kappa$. Since $Y$ is $\kappa$-Scheepers, for each $\alpha< \kappa$, there is $\mathcal{V}_\alpha \in [\mathcal{U}_\alpha]^{<\kappa}$ such that $\mathcal{V} = \{\bigcup \mathcal{V}_\alpha : \alpha< \kappa\} \in \Omega_\kappa(Y)$. Define $c\in \kappa^\kappa$ by
\[c(\beta) = \begin{cases}
\sup\{\xi+1 : U_\xi^{(\alpha)} \in \mathcal{V}_\alpha\}, & \mbox{if } \beta = a_\alpha \text{ for some } \alpha< \kappa\\
0, & \mbox{if } \beta\in \kappa \setminus A.
\end{cases}\]
Moreover, $\bigcup \mathcal{V}_\alpha \subseteq \{y\in Y : y(a_\alpha) < c(a_\alpha)\}$ for all $\alpha< \kappa$. Choose $d\in \kappa^\kappa$ given by $d(\alpha) = c(\alpha)+1$ for all $\alpha< \kappa$. Since $\maxk(Y)$ is dominating, there is $F\in [Y]^{<\kappa}$ such that \[d\leq_\kappa^* \left(\alpha \mapsto \sup \{f(\alpha) : f\in F\}\right).\] Thus there is $E\in [\kappa]^{<\kappa}$ such that for every $\beta\in \kappa \setminus E$, $c(\beta)+1 \leq \sup \{f(\beta) : f\in F\}$. It follows that for every $\alpha< \kappa$ with $a_\alpha \notin E$, we have $F\not \subseteq \bigcup \mathcal{V}_\alpha$.

For each $\alpha< \kappa$ with $a_\alpha \in E$, choose $y_\alpha \in
Y \setminus \bigcup \mathcal{V}_\alpha$, which is possible since $\mathcal{V}$ is a $\kappa$-cover. Then $H = F\cup \{y_\alpha : a_\alpha \in E\}$ has size less than $\kappa$, but $H\not \subseteq \bigcup \mathcal{V}_\alpha$ for every $\alpha< \kappa$. This contradicts $\mathcal{V} \in \Omega_\kappa(Y)$. Hence $\maxk(Y)$ is not dominating.

Conversely, suppose that for every continuous image $Y$ of $X$ into $\kappa^\kappa$, the family $\maxk(Y)$ is not dominating. Let $(\mathcal{U}_\alpha : \alpha< \kappa) \in \mathcal{O}_\kappa(X)^\kappa$. For each $\alpha< \kappa$, using zero-dimensionality and $\kappa$-Lindel\"{o}fness, choose a clopen refinement $\{C_\xi^{(\alpha)} : \xi< \kappa\}$ of $\mathcal{U}_\alpha$ which covers $X$. For each $\xi< \kappa$, choose $U_\xi^{(\alpha)} \in \mathcal{U}_\alpha$ such that $C_\xi^{(\alpha)} \subseteq U_\xi^{(\alpha)}$. Put $D_\xi^{(\alpha)} = \bigcup_{\eta \leq \xi} C_\eta^{(\alpha)}$.

Define $\Phi : X\to \kappa^\kappa$ by $\Phi(x)(\alpha) = \min \{\xi< \kappa : x\in D_\xi^{(\alpha)}\}$. Then $\Phi$ is continuous. Let $Y = \Phi[X]$. By assumption, $\maxk(Y)$ is not dominating. There exists $g\in \kappa^\kappa$ such that $g\not \leq_\kappa^* h$ for every $h\in \maxk(Y)$. For each $\alpha< \kappa$, let $\mathcal{V}_\alpha = \{U_\xi^{(\alpha)} : \xi\leq g(\alpha)\}$. It can be shown that $\{\bigcup \mathcal{V}_\alpha : \alpha< \kappa\}\in \Omega_{\kappa}(X)$. Hence $X$ is $\kappa$-Scheepers.
\end{proof}

\begin{Th}
\label{thm2}
Assume $\kappa^{<\kappa} = \kappa$. Then $\non(\kappa\text{-Scheepers}) = \dlk$.
\end{Th}

\begin{proof}
Let $X\subseteq P(\kappa)$ with $|X|< \dlk$. Since $P(\kappa)$ has a base of size $\kappa$, the space $X$ is $\kappa$-Lindel\"{o}f. Let $Y$ be a continuous image of $X$ into $\kappa^\kappa$. Then $|Y|\leq |X|< \dlk$. By Lemma~\ref{lemma8}, $\maxk(Y)$ is not dominating. Hence by Lemma~\ref{lemma7}, $X$ is $\kappa$-Scheepers. Therefore, $\dlk \leq \non(\kappa\mbox{-Scheepers})$.

For the reverse inequality, let $A\subseteq \kappa^\kappa$ be such that $|A| = \dlk$ and $\maxk(A)$ is dominating (see Lemma~\ref{lemma8}). Since $\maxk(A)$ is dominating, Lemma~\ref{lemma7} implies that $A$ is not $\kappa$-Scheepers. Thus $\non(\kappa\mbox{-Scheepers}) \leq |A| = \dlk$. Consequently, $\non(\kappa\mbox{-Scheepers}) = \dlk$.
\end{proof}

\begin{Cor}
\label{cor3}
Assume $\kappa^{<\kappa} = \kappa$.
\begin{enumerate}[wide=0pt,label={\upshape(\arabic*)},leftmargin=*]
  \item If $\lambda^{<\kappa} < \dk$ for every $\lambda < \dk$, then $\non(\kappa\mbox{-Scheepers}) = \dk$.
  \item If $\dk = \kappa^+$, then $\non(\kappa\mbox{-Scheepers}) = \dk$.
\end{enumerate}
\end{Cor}

\begin{proof}
The proof of $(1)$ follows from Lemma~\ref{lemma902}; we prove only $(2)$. It is enough to show that $\dlk = \dk$. By Lemma~\ref{lemma901}, $\dlk \leq \dk$. It remains to prove the reverse inequality. Let $Y\subseteq \kappa^\kappa$ with $|Y|< \dk =\kappa^+$. Then $|Y| \leq \kappa$. Therefore, \[|\maxk(Y)|
\leq |[Y]^{<\kappa}| \leq |Y|^{<\kappa} \leq \kappa^{<\kappa} = \kappa.\] Since $\kappa< \kappa^+ = \dk$, the family $\maxk(Y)$ cannot be dominating. Hence by Lemma~\ref{lemma8}, $\dlk \geq \dk$. Thus $\dlk = \dk$.
\end{proof}

\begin{Cor}
\label{cor4}
Assume $\kappa^{<\kappa} = \kappa$. If $\bk< \dk$, then there exists a space satisfying $\S1k(\Gamma_\kappa, \Omega_\kappa)$ which is not $\kappa$-Hurewicz.
\end{Cor}

\begin{proof}
Since $\non(\kappa\mbox{-Hurewicz}) = \bk$ (see Theorem~\ref{thm14}), there exists a subspace $X$ of $P(\kappa)$ such that $|X| = \bk$ and $X$ is not $\kappa$-Hurewicz. Now \cite[Proposition~1.9(ii) and Theorem~1.12]{mrs05} give that if $\bk< \dk$, then $\bk< \dlk$. It follows that $|X| < \dlk = \non(\S1k(\Gamma_\kappa, \Omega_\kappa))$ (see Theorem~\ref{lemma11}), and hence $X$ satisfies  $\S1k(\Gamma_\kappa, \Omega_\kappa)$.
\end{proof}

Following Khomskii et al. \cite[Definition~2.15]{klls16}, for a regular cardinal $\kappa$, $f, g\in \kappa^\kappa$ are said to be eventually different if there exists $\beta< \kappa$ such that $f(\alpha) \neq g(\alpha)$ for every $\alpha \geq \beta$; equivalently, the set $\{\alpha< \kappa : f(\alpha) = g(\alpha)\}$ has size less than $\kappa$. A family $E \subseteq \kappa^\kappa$ is said to be eventually different if for each $g\in \kappa^\kappa$, there exists $f\in E$ which is eventually different from $g$. The cardinal \[\ink = \min \{|E| :  E \subseteq \kappa^\kappa \text{ is an eventually different family}\}\] is called the inequality number.

We obtain the following continuous-image characterization of the $\kappa$-Rothberger property.

\begin{Lemma}
\label{lemma10}
Assume that $\kappa$ is regular. Let $X$ be a zero-dimensional $\kappa$-Lindel\"{o}f $P_\kappa$-space. Then $X$ is $\kappa$-Rothberger if and only if no continuous image of $X$ into $\kappa^\kappa$ is an eventually different family.
\end{Lemma}

\begin{proof}
Suppose that $X$ is $\kappa$-Rothberger. Let $\Phi : X\to \kappa^\kappa$ be continuous and put $Y = \Phi[X]$. We show that $Y$ is not an eventually different family. Suppose, on the contrary, that $Y$ is eventually different. For $I\in [\kappa]^\kappa$, $\beta\in I$, and $\xi< \kappa$, put $C_{\beta,\xi} = \{f\in Y : f(\beta) = \xi\}$, and let $\mathcal C_I = \{C_{\beta,\xi} : \beta\in I, \xi< \kappa\}$. Each $C_{\beta,\xi}$ is clopen in $Y$, and $\mathcal{C}_I$ is an open cover of $Y$.

We claim that $|\mathcal{C}_I| = \kappa$ for every $I\in [\kappa]^\kappa$. Suppose that $|\mathcal{C}_I| < \kappa$ for some $I\in [\kappa]^\kappa$. For each $C\in \mathcal{C}_I$, put $D_C = \{\beta\in I : (\exists \xi< \kappa)\ C_{\beta,\xi} = C\}$, and let $\mathcal{D} = \{C\in \mathcal{C}_I : |D_C| = \kappa\}$. Then $\mathcal{D}$ covers $Y$. Since $|\mathcal{D}| < \kappa$ and $|D_C| = \kappa$ for every $C\in \mathcal{D}$, a recursion of length $\kappa$ yields pairwise disjoint sets $J_C\in [D_C]^\kappa$, $C\in \mathcal{D}$. For each $C\in \mathcal{D}$ and each $\beta\in J_C$, choose $\xi_\beta < \kappa$ such that $C_{\beta, \xi_\beta} = C$. Since the sets $J_C$ are pairwise disjoint, we may define $g\in \kappa^\kappa$ by $g(\beta) = \xi_\beta$ for all $\beta\in \bigcup_{C\in \mathcal{D}} J_C$, and define $g$ arbitrarily elsewhere. Let $f\in Y$. Choose $C\in \mathcal{D}$ such that $f\in C$. For each $\beta\in J_C$, $C = C_{\beta, \xi_\beta}$, and hence $f(\beta) = \xi_\beta = g(\beta)$. Therefore, $|\{\beta< \kappa : f(\beta) = g(\beta)\}| \geq |J_C| = \kappa$, which contradicts $Y$ being eventually different. Hence $|\mathcal{C}_I| = \kappa$ for every $I\in [\kappa]^\kappa$. Thus $\mathcal{C}_I \in \mathcal{O}_\kappa(Y)$ for every $I\in [\kappa]^\kappa$.

Choose a partition $\kappa = \bigcup_{\alpha< \kappa} I_\alpha$ such that $I_\alpha$'s are pairwise disjoint and $|I_\alpha| = \kappa$ for every $\alpha< \kappa$. Choose also a partition $\kappa = \bigcup_{\gamma< \kappa} A_\gamma$ such that $A_\gamma$'s are pairwise disjoint and $|A_\gamma| = \kappa$ for every $\gamma< \kappa$. For each $\gamma< \kappa$, after reindexing the sequence $(\mathcal{C}_{I_\alpha} : \alpha\in A_\gamma)$ by $\kappa$, apply
$\S1k(\mathcal{O}_\kappa, \mathcal{O}_\kappa)$. Thus for each $\alpha\in A_\gamma$, choose $C_{\beta_\alpha,\xi_\alpha} \in \mathcal{C}_{I_\alpha}$ so that $Y = \bigcup_{\alpha\in A_\gamma} C_{\beta_\alpha, \xi_\alpha}$. Since $\beta_\alpha\in I_\alpha$ and the sets $I_\alpha$ are pairwise disjoint, the ordinals $\beta_\alpha$, $\alpha< \kappa$, are pairwise distinct. Define
$h\in \kappa^\kappa$ by $h(\beta_\alpha) = \xi_\alpha$ for all $\alpha< \kappa$,
and extend $h$ arbitrarily to the remaining coordinates. Let $f\in Y$. For every $\gamma< \kappa$, there exists $\alpha\in A_\gamma$ such that $f\in C_{\beta_\alpha, \xi_\alpha}$. Thus $f(\beta_\alpha) = \xi_\alpha = h(\beta_\alpha)$. Since the sets $A_\gamma$, $\gamma< \kappa$, are pairwise disjoint, these give $\kappa$ many distinct coordinates at which $f$ and $h$ agree. Hence the size of the set $\{\beta< \kappa : f(\beta) = h(\beta)\}$ is $\kappa$. This holds for every $f\in Y$, contradicting the assumption that $Y$ is eventually different. Therefore, $Y$ is not an eventually different family.

Conversely, suppose that no continuous image of $X$ into $\kappa^\kappa$ is an eventually different family. Let $(\mathcal{U}_\alpha : \alpha< \kappa) \in \mathcal{O}_\kappa(X)^\kappa$. Choose a partition $\kappa = A\cup B$ such that $|A| = |B| = \kappa$, and write $A = \{a_\delta : \delta< \kappa\}$. For each $\alpha\in A$, since $X$ is zero-dimensional, $\mathcal{U}_\alpha$ has a clopen refinement. Since $X$ is $\kappa$-Lindel\"{o}f, we may choose a clopen refinement of cardinality at most $\kappa$. Repeating members if necessary, write it as $\{C_\xi^{(\alpha)} : \xi< \kappa\}$. For each $\xi< \kappa$, choose $U_\xi^{(\alpha)} \in \mathcal{U}_\alpha$ such that $C_\xi^{(\alpha)} \subseteq U_\xi^{(\alpha)}$, and let $P_\xi^{(\alpha)} = C_\xi^{(\alpha)} \setminus \bigcup_{\eta< \xi} C_\xi^{(\alpha)}$.

Define $\Phi : X\to \kappa^\kappa$ by $\Phi(x)(\delta) = \xi$ if and only if $x\in P_\xi^{(a_\delta)}$. Then $\Phi$ is continuous. Indeed, for every $s\in \kappa^{<\kappa}$, $\Phi^{-1}([s]) = \bigcap_{\delta\in \dom(s)} P_{s(\delta)}^{(a_\delta)}$, which is open since $|\dom(s)|< \kappa$ and $X$ is a $P_\kappa$-space. Put $Y = \Phi[X]$. By assumption, there exists $g\in \kappa^\kappa$ such that for every $x\in X$, $\{\delta< \kappa : \Phi(x)(\delta) = g(\delta)\}$ has size $\kappa$. For each $\delta< \kappa$, choose $U_{a_\delta} = U_{g(\delta)}^{(a_\delta)} \in \mathcal{U}_{a_\delta}$. Let $x\in X$. Choose $\delta< \kappa$ such that $\Phi(x)(\delta) = g(\delta)$. Then $x\in P_{g(\delta)}^{(a_\delta)} \subseteq U_{g(\delta)}^{(a_\delta)} = U_{a_\delta}$. Therefore, $X = \bigcup_{\alpha\in A} U_\alpha$.

It remains to ensure that the family of all selected members has size exactly $\kappa$, as required by the definition of $\mathcal{O}_\kappa$. Write $B = \{b_\xi : \xi< \kappa\}$. Recursively choose $U_{b_\xi} \in \mathcal{U}_{b_\xi}$ so that $U_{b_\xi} \notin \{U_{b_\eta} : \eta< \xi\}$. This is possible since $|\mathcal{U}_{b_\xi}| = \kappa$ and fewer than $\kappa$ members have been chosen at stage $\xi$. Consequently, $\{U_\alpha : \alpha< \kappa\}$ has size $\kappa$. This yields $\{U_\alpha : \alpha< \kappa\}\in \mathcal{O}_\kappa(X)$. Therefore, $X$ satisfies $\S1k(\mathcal{O}_\kappa, \mathcal{O}_\kappa)$.
\end{proof}

Arguing as in the proof of Theorem~\ref{thm2} and using the preceding lemma, we obtain the following result.

\begin{Th}
\label{thm9}
Assume $\kappa^{<\kappa} = \kappa$. Then $\non(\kappa\text{-Rothberger}) = \ink$.
\end{Th}

Let $\kappa$ be a regular cardinal. Recall that a subset of $\kappa^\kappa$ is $\kappa$-meager if it is the union of at most $\kappa$ many nowhere dense subsets of $\kappa^\kappa$. Let $\mathcal{M}_\kappa$ denote the family of all $\kappa$-meager subsets of $\kappa^\kappa$. The covering number of $\mathcal{M}_\kappa$ is defined by \[\cov(\mathcal{M}_\kappa) = \min \left\{ |\mathcal{A}| : \mathcal{A} \subseteq \mathcal{M}_\kappa \text{ and } \bigcup \mathcal{A} = \kappa^\kappa\right\}.\] Thus $\cov(\mathcal{M}_\kappa)$ is the least cardinality of a family of $\kappa$-meager subsets of $\kappa^\kappa$ whose union is $\kappa^\kappa$.

Landver \cite{lander92} proved that $\ink = \cov(\mathcal{M}_\kappa)$ whenever $\kappa$ is strongly inaccessible. Consequently, we have the following.

\begin{Cor}
\label{cor6}
Assume that $\kappa$ is strongly inaccessible. Then $\non(\kappa\mbox{-Rothberger}) = \ink = \cov(\mathcal{M}_\kappa)$.
\end{Cor}

\begin{Cor}
\label{cor5}
Assume that $\kappa$ is a successor cardinal satisfying $\kappa^{<\kappa} = \kappa$. If $\dlk< \dk$, then there exists a $\kappa$-Rothberger space which is not $\kappa$-Scheepers.
\end{Cor}

\begin{proof}
Since $\kappa$ is a successor cardinal satisfying $\kappa^{<\kappa} = \kappa$, $\ink = \dk$ (see \cite{klls16}). Hence $\dlk< \ink$. Since $\non(\kappa\mbox{-Scheepers}) = \dlk$ (see Theorem~\ref{thm2}), there exists $X\subseteq P(\kappa)$ such that $|X| = \dlk$ and $X$ is not $\kappa$-Scheepers. By Theorem~\ref{thm9}, $X$ is $\kappa$-Rothberger.
\end{proof}

\begin{Rem}
\label{rem8}
By \cite[Theorem~0.1]{mrs05}, assuming GCH (i.e. $2^\lambda = \lambda^+$ for every infinite cardinal $\lambda$), there is a forcing extension in which $\dlk< \dk$. Therefore, it is consistent with ZFC that there exists a $\kappa$-Rothberger space which is not $\kappa$-Scheepers.
\end{Rem}

\begin{Th}
\label{lemma11}
Assume $\kappa^{<\kappa} = \kappa$. Then $\non(\S1k(\Gamma_\kappa, \Omega_\kappa)) = \dlk$.
\end{Th}

\begin{proof}
By Theorem~\ref{thm2}, $\non(\S1k(\Gamma_\kappa, \Omega_\kappa)) \leq \dlk$. We now show that $\non(\S1k(\Gamma_\kappa, \Omega_\kappa)) \geq \dlk$. Let $X$ be a subspace of $P(\kappa)$ with $|X|< \dlk$, and let $(\mathcal{U}_\alpha : \alpha< \kappa) \in \Gamma_\kappa(X)^\kappa$. For each $\alpha< \kappa$, write $\mathcal{U}_\alpha = \{U_\xi^{(\alpha)} : \xi< \kappa\}$ without repetitions. For each $x\in X$, define $f_x\in \kappa^\kappa$ by \[f_x(\alpha) = \sup \{\xi+1 : x\notin U_\xi^{(\alpha)}\},\] where $\sup \emptyset = 0$. Put $Y = \{f_x : x\in X\}$. Since $|Y|< \dlk$, the family $\maxk(Y)$ is not dominating. Hence there exists $g\in \kappa^\kappa$ such that $g\not \leq_\kappa^* h$ for every $h\in \maxk(Y)$.

Recursively choose $\xi_\alpha< \kappa$ for $\alpha< \kappa$ such that $\xi_\alpha \geq g(\alpha)$ and $U_{\xi_\alpha}^{(\alpha)} \neq X$, and the sets $U_{\xi_\alpha}^{(\alpha)}$, $\alpha< \kappa$, are pairwise distinct. Then $\{U_{\xi_\alpha}^{(\alpha)} : \alpha< \kappa\} \in \Omega_\kappa(X)$. Therefore, $X$ satisfies $\S1k(\Gamma_\kappa, \Omega_\kappa)$. It follows that $\non(\S1k(\Gamma_\kappa, \Omega_\kappa)) \geq \dlk$. Thus $\S1k(\Gamma_\kappa, \Omega_\kappa) = \dlk$.
\end{proof}

\begin{Th}
\label{lemma12}
Assume $\kappa^{<\kappa} = \kappa$. Then $\non(\S1k(\Gamma_\kappa, \mathcal{O}_\kappa)) = \dk$.
\end{Th}

\begin{proof}
The proof is analogous to that of Theorem~\ref{lemma11} and is therefore omitted.
\end{proof}

\begin{Th}
\label{thm10}
Assume that $\kappa$ is weakly compact. Then the $\kappa$-Cantor space $2^\kappa$ satisfies $\Sk(\Omega_\kappa, \Omega_\kappa)$, but does not satisfy $\S1k(\Gamma_\kappa, \mathcal{O}_\kappa)$.
\end{Th}

\begin{proof}
Since for every cardinal $0< \lambda< \kappa$, $(2^\kappa)^\lambda$ endowed with the box topology is $\kappa$-compact, $2^\kappa$ satisfies $\Sk(\Omega_\kappa, \Omega_\kappa)$. We next prove that $2^\kappa$ does not satisfy $\S1k(\Gamma_\kappa, \mathcal{O}_\kappa)$.

Partition $\kappa$ into pairwise disjoint sets $\kappa = \bigcup_{\alpha< \kappa} I_\alpha$, $|I_\alpha| = \kappa$. For each $\alpha< \kappa$, fix a bijection $e_\alpha : \kappa\to I_\alpha$. For each $\xi< \kappa$, put \[B_\xi^{(\alpha)} = \{x\in 2^\kappa : x(e_\alpha(\eta)) = 0 \text{ for every } \eta< \xi \text{ and } x(e_\alpha(\xi)) = 1\}.\] Each $B_\xi^{(\alpha)}$ is a nonempty clopen subset of $2^\kappa$, and for each fixed $\alpha$, the family $\{B_\xi^{(\alpha)} : \xi< \kappa\}$ is pairwise disjoint. Define $\mathcal{U}_\alpha = \{2^\kappa \setminus B_\xi^{(\alpha)} : \xi< \kappa\}$. Then $\mathcal{U}_\alpha \in \Gamma_\kappa(2^\kappa)$.

Now choose arbitrarily $U_\alpha \in \mathcal{U}_\alpha$, $\alpha< \kappa$. Write $U_\alpha = 2^\kappa \setminus B_{\xi_\alpha}^{(\alpha)}$ for some $\xi_\alpha < \kappa$. Since the sets $I_\alpha$, $\alpha< \kappa$, are pairwise disjoint, we may define $x\in 2^\kappa$ so that for every $\alpha< \kappa$,
\[x(e_\alpha(\eta)) = \begin{cases}
0, & \mbox{if } \eta< \xi_\alpha\\
1, & \mbox{if } \eta = \xi_\alpha.
\end{cases}\]
The values of $x$ on the remaining coordinates may be chosen arbitrarily. Thus $x\in B_{\xi_\alpha}^{(\alpha)}$ for every $\alpha< \kappa$. Consequently, $x\notin U_\alpha$ for every $\alpha< \kappa$. Thus $\{U_\alpha : \alpha< \kappa\}$ does not cover $2^\kappa$. Hence $2^\kappa$ does not satisfy $\S1k(\Gamma_\kappa, \mathcal{O}_\kappa)$.
\end{proof}

\begin{Cor}
\label{cor7}
Assume that $\kappa$ is weakly compact. Then the $\kappa$-Cantor space $2^\kappa$ satisfies $\Uk(\mathcal{O}_\kappa, \Gamma_\kappa)$, but does not satisfy $\S1k(\Gamma_\kappa, \Gamma_\kappa)$.
\end{Cor}

The following observation is a consequence of Theorems~\labelcref{thm16,thm17}.

\begin{Prop}
\label{prop2}
Assume that $\kappa$ is regular. If $\pk< \bk$, then $\S1k(\Gamma_\kappa, \Gamma_\kappa)$ does not imply $\S1k(\Omega_\kappa, \Gamma_\kappa)$.
\end{Prop}

\begin{Rem}
\label{rem2}
In \cite[Theorem~4.8 and Observation~4.2]{fms22}, the authors showed that, assuming GCH, it is consistent that $\pk = \kappa^+ < \bk = 2^\kappa$ for a cardinal $\kappa$ satisfying $\kappa^{<\kappa} = \kappa$.
\end{Rem}

Theorems~\labelcref{thm9,lemma11} give the following observation.

\begin{Prop}
\label{prop3}
Assume $\kappa^{<\kappa} = \kappa$. If $\ink< \dlk$, then $\S1k(\Gamma_\kappa, \Omega_\kappa)$ does not imply $\S1k(\mathcal{O}_\kappa, \mathcal{O}_\kappa)$.
\end{Prop}

\begin{Rem}
\label{rem3}
Assume that the existence of a supercompact cardinal is consistent. Then it is consistent that there exists a strongly inaccessible cardinal $\kappa$ such that $\ink< \dlk$.
\end{Rem}

\begin{proof}
Start with a model of GCH containing a supercompact cardinal $\kappa$. By Theorem~2.1 of Shelah \cite{shelah20}, there is a cardinal-preserving forcing extension in which $\kappa$ remains supercompact and $\cov(\mathcal{M}_\kappa) = \kappa^+$ and $\dk = \kappa^{++}$. Since $\kappa$ remains supercompact, it is strongly inaccessible in the extension, and hence $(\kappa^+)^{<\kappa} = \kappa^+$.

Since $\kappa$ is strongly inaccessible, $\ink = \cov(\mathcal{M}_\kappa)$ (see \cite{lander92}), and so $\ink = \kappa^+$. We now show that $\dlk = \kappa^{++}$. Certainly $\dlk \leq \dk = \kappa^{++}$. Suppose that $D\subseteq \kappa^\kappa$ is such that $\maxk(D)$ is dominating and $|D|< \kappa^{++}$. Then $|D|\leq \kappa^+$. It follows that $|\maxk(D)| \leq |D|^{<\kappa} \leq( \kappa^+)^{<\kappa} = \kappa^+$. This contradicts $\dk = \kappa^{++}$. Hence by Lemma~\ref{lemma8}, $\dlk = \kappa^{++}$. Therefore, $\ink = \kappa^+ < \kappa^{++} = \dlk$.
\end{proof}

\begin{Th}
\label{thm11}
Assume $\kappa^{<\kappa} = \kappa$. If $\pk = \kappa^+ < \cov(\mathcal{M}_\kappa)$, then $\S1k(\Omega_\kappa, \Omega_\kappa)$ does not imply $\S1k(\Omega_\kappa, \Gamma_\kappa)$.
\end{Th}

\begin{proof}
Since $\non(\S1k(\Omega_\kappa, \Gamma_\kappa)) = \pk$ (see Theorem~\ref{thm17}), there exists a space $X\subseteq P(\kappa)$ with $|X| = \pk = \kappa^+$ such that $X$ does not satisfy $\S1k(\Omega_\kappa, \Gamma_\kappa)$. We now show that $X$ satisfies $\S1k(\Omega_\kappa, \Omega_\kappa)$. Since $|[X]^{<\kappa}| = \max \{\kappa^+, \kappa^{<\kappa}\}$ and $\kappa^{<\kappa} = \kappa$, we get $|[X]^{<\kappa}| = \kappa^+ < \cov(\mathcal{M}_\kappa)$. Let $(\mathcal{U}_\alpha : \alpha< \kappa)\in \Omega_\kappa(X)^\kappa$, and write $\mathcal{U}_\alpha = \{U_\xi^{(\alpha)} : \xi< \kappa\}$. For every $F\in [X]^{<\kappa}$, define \[D_F = \left\{g\in \kappa^\kappa : (\exists \alpha< \kappa)\, F\subseteq U_{g(\alpha)}^{(\alpha)}\right\}.\] Then $\kappa^\kappa \setminus D_F$ is nowhere dense for every $F\in [X]^{<\kappa}$. Since $|[X]^{<\kappa}| < \cov(\mathcal{M}_\kappa)$, the family $\{\kappa^\kappa \setminus D_F : F\in [X]^{<\kappa}\}$ cannot cover $\kappa^\kappa$. Hence $\bigcap_{F\in [X]^{<\kappa}} D_F\neq \emptyset$. Choose $g\in \bigcap_{F\in [X]^{<\kappa}} D_F$. Then $\mathcal{V} = \{U_{g(\alpha)}^{(\alpha)} : \alpha< \kappa\}$ is a $\kappa$-cover of $X$, and therefore, $X$ satisfies $\S1k(\Omega_\kappa, \Omega_\kappa)$.
\end{proof}

\begin{Rem}
\label{rem5}
By Theorem~4.2 of Honzik and Stejskalov\'{a} \cite{hs25}, it is consistent, relative to the consistency of the existence of a supercompact cardinal with a weakly compact cardinal above it, that there exists a strongly inaccessible cardinal $\kappa$ such that $\pk = \kappa^+ < \cov(\mathcal{M}_\kappa)$. \end{Rem}

\begin{Lemma}
\label{lemma15}
Assume $\kappa^{<\kappa} = \kappa$. If $X\subseteq 2^\kappa$ satisfies $\S1k(\Omega_\kappa, \Omega_\kappa)$, then $X^2$ is $\kappa$-Rothberger.
\end{Lemma}

\begin{proof}
Let $(\mathcal{U}_\xi : \xi< \kappa) \in \mathcal{O}_\kappa(X^2)^\kappa$. Partition $\kappa$ into pairwise disjoint sets $R$ and $I_\alpha$, $\alpha< \kappa$ such that $|R| = |I_\alpha| = \kappa$ for every $\alpha< \kappa$. Let $\mathcal{B}$ be the standard base of $X$ inherited from $2^\kappa$. Since $\kappa^{<\kappa}=\kappa$, we have $|\mathcal B|\leq\kappa$. Let \[\mathcal{C} = \left\{\bigcup \mathcal{A} : \mathcal{A} \in [\mathcal{B}]^{<\kappa}\right\}.\] Then $|\mathcal{C}| \leq \kappa$.

Fix $\alpha< \kappa$. For each nonempty $F\in [X]^{<\kappa}$, enumerate $F^2$ as $F^2 = \{p_\eta : \eta< \mu\}$, $\mu< \kappa$. Choose pairwise distinct indices $\xi_\eta \in I_\alpha$, $\eta< \mu$, and for each $\eta< \mu$, choose $U_{\xi_\eta} \in \mathcal{U}_{\xi_\eta}$ such that $p_\eta \in U_{\xi_\eta}$. Put $G = \bigcup_{\eta< \mu} U_{\xi_\eta}$. Then $G$ is an open subset of $X^2$ containing $F^2$. For each $(x,y)\in F^2$, choose $B_{x,y}, C_{x,y} \in \mathcal{B}$ such that $x\in B_{x,y}$, $y\in C_{x,y}$, and $B_{x,y}\times C_{x,y} \subseteq G$. For $x\in F$, put $N_x = \bigcap_{y\in F} \bigl(B_{x,y} \cap C_{y,x}\bigr)$. Since $|F|< \kappa$ and $2^\kappa$ is a $P_\kappa$-space, $N_x$ is open in $X$. Moreover, we may choose $N_x\in \mathcal{B}$. Set $W = \bigcup_{x\in F} N_x$. Then $W\in \mathcal{C}$, $F\subseteq W$, and $W^2\subseteq G$.

If, for some $\alpha< \kappa$ and some $F\in [X]^{<\kappa}$, we obtain $W = X$, then $X^2 = G = \bigcup_{\eta< \mu} U_{\xi_\eta}$. Thus fewer than $\kappa$ choices from distinct members of $(\mathcal{U}_\xi : \xi< \kappa)$ already cover $X^2$. Choosing one member from every remaining $\mathcal{U}_\xi$, and making the remaining choices pairwise distinct yields immediately a member of $\mathcal{O}_\kappa(X^2)$.

We may therefore assume that all the sets $W$ obtained above are proper
subsets of $X$. For each $\alpha< \kappa$, let $\mathcal{W}_\alpha$ be the family of all proper $W\in \mathcal{C}$ for which there exist $J\in [I_\alpha]^{<\kappa}$ and choices $U_\xi\in \mathcal{U}_\xi$, $\xi\in J$ such that $W^2 \subseteq \bigcup_{\xi\in J} U_\xi$. Clearly, $\mathcal{W}_\alpha \in \Omega_\kappa(X)$, $\alpha< \kappa$. There are $W_\alpha \in \mathcal{W}_\alpha$, $\alpha< \kappa$ such that $\{W_\alpha : \alpha< \kappa\}\in \Omega_\kappa(X)$. For every $\alpha< \kappa$, choose $J_\alpha \in [I_\alpha]^{<\kappa}$ and $U_\xi^{(\alpha)} \in \mathcal{U}_\xi$, $\xi\in J_\alpha$ such that $W_\alpha^2 \subseteq \bigcup_{\xi\in J_\alpha} U_\xi^{(\alpha)}$. We claim that $\bigcup_{\alpha< \kappa} \{U_\xi^{(\alpha)} : \xi\in J_\alpha\}$ covers $X^2$. Indeed, let $(x,y)\in X^2$. There is $\alpha< \kappa$ such that $(x,y)\in W_\alpha^2 \subseteq \bigcup_{\xi\in J_\alpha} U_\xi^{(\alpha)}$. Since the sets $I_\alpha$ are pairwise disjoint, the above choices come from distinct members of $(\mathcal{U}_\xi : \xi< \kappa)$. Finally, choose one member from each remaining $\mathcal U_\xi$. Using the unused set $R$, the additional choices may be made pairwise distinct, and therefore the resulting selected family has size $\kappa$. It is an open cover
of $X^2$. Thus $X^2$ satisfies $\S1k(\mathcal{O}_\kappa, \mathcal{O}_\kappa)$.
\end{proof}

Let $\kappa< \lambda \leq 2^\kappa$. A set $L\subseteq 2^\kappa$ is said to be a $\lambda \mbox{-} \kappa$-Lusin set \cite{kw20} if $|L|\geq \lambda$ and for every $\kappa$-meager set $M\subseteq 2^\kappa$, $|M\cap L|< \lambda$. A set $A\subseteq 2^\kappa$ is said to be $\kappa$-strongly measure zero (or $\kappa$-strongly null) \cite{hs01} (see also \cite{kw20}) if for every $(\xi_\alpha : \alpha< \kappa)\in \kappa^\kappa$, there exists $(x_\alpha : \alpha< \kappa)$ such that $x_\alpha \in 2^{\xi_\alpha}$, $\alpha< \kappa$ and $A\subseteq \bigcup_{\alpha< \kappa} [x_\alpha]$, where $[x_\alpha] = \{f\in 2^\kappa : f\restriction \dom(x_\alpha) = x_\alpha\}$.

\begin{Th}
\label{thm12}
Assume $\kappa^{<\kappa} = \kappa$ and GCH at $\kappa$. Then there exists a $\kappa^+$-$\kappa$-Lusin set $L\subseteq 2^\kappa$ such that $L$ is $\kappa$-Rothberger, but does not satisfy $\S1k(\Omega_\kappa, \Omega_\kappa)$.
\end{Th}

\begin{proof}
Assume GCH at $\kappa$. By \cite[Proposition~8]{kw20}, there exists a $\kappa^+$-$\kappa$-Lusin set $L\subseteq 2^\kappa$ such that under the canonical identification $\Phi: 2^\kappa \times 2^\kappa \to 2^\kappa$, we have $\Phi[L^2]$ is not $\kappa$-strongly measure zero. By \cite[Corollary~46]{kw20}, $L$ is $\kappa$-Rothberger.

Suppose, towards a contradiction, that $L$ satisfies $\S1k(\Omega_\kappa, \Omega_\kappa)$. By Lemma~\ref{lemma15}, $L^2$ is $\kappa$-Rothberger, and hence $\Phi[L^2] \subseteq 2^\kappa$ is also $\kappa$-Rothberger. Therefore, by
\cite[Proposition~43]{kw20}, $\Phi[L^2]$ is $\kappa$-strongly measure zero, which contradicts the choice of $L$. Therefore, $L$ does not satisfy $\S1k(\Omega_\kappa, \Omega_\kappa)$.
\end{proof}

\subsection{The diagram and open problems}
\label{sec-sdu}
Thus if $\kappa$ is a strongly inaccessible cardinal, then for every $P_\kappa$-space, the results discussed above yield the Scheepers diagram
in the uncountable setting (Figure~\ref{sdu}). In Figure~\ref{sdu},
a question mark above a horizontal arrow indicates that it is unknown whether
the corresponding implication holds, whereas a question mark below a horizontal
arrow, or to the right of a vertical arrow, indicates that it is unknown whether
the reverse implication holds. The cardinal number displayed below each property
indicates its critical cardinality, while a question mark in place of a cardinal number indicates that the critical cardinality is unknown.

\begin{figure}[h]
\begin{adjustbox}{max width = \textwidth, max height = \textheight, keepaspectratio, center}
\begin{tikzcd}[column sep=6ex, row sep=4ex, arrows = {crossing over}]
& \sr{$\Uk(\mathcal{O}_\kappa, \Gamma_\kappa)$}{$\bk$}{} \arrow[r, "\textbf{?}"] \arrow[rr, bend left=20] & \sr{$\Sk(\Gamma_\kappa, \Omega_\kappa)$}{$\dlk$}{} \arrow[r, "\textbf{?}", swap] & \sr{$\Uk(\mathcal{O}_\kappa, \Omega_\kappa)$}{$\dlk$}{} \arrow[r] & \sr{$\Sk(\mathcal{O}_\kappa, \mathcal{O}_\kappa)$}{$\dk$}{}
\\
\sr{$\S1k(\Gamma_\kappa, \Gamma_\kappa)$}{$\bk$}{} \arrow[r] \arrow[ru] & \sr{$\S1k(\Gamma_\kappa, \Omega_\kappa)$}{$\dlk$}{} \arrow[rr] \arrow[ru] && \sr{$\S1k(\Gamma_\kappa, \mathcal{O}_\kappa)$}{$\dk$}{} \arrow[ru] &
\\
&& \sr{$\Sk(\Omega_\kappa, \Omega_\kappa)$}{{\bf ?}}{} \arrow[uu, "\textbf{?}", swap, sloped, pos=0.75] &&
\\
\sr{$\S1k(\Omega_\kappa, \Gamma_\kappa)$}{$\pk$}{} \arrow[r] \arrow[uu] & \sr{$\S1k(\Omega_\kappa, \Omega_\kappa)$}{{\bf ?}}{} \arrow[rr] \arrow[ru] \arrow[uu] && \sr{$\S1k(\mathcal{O}_\kappa, \mathcal{O}_\kappa)$}{$\ink$}{} \arrow[uu] &
\end{tikzcd}
\end{adjustbox}
\caption{Scheepers diagram in the uncountable setting}
\label{sdu}
\end{figure}
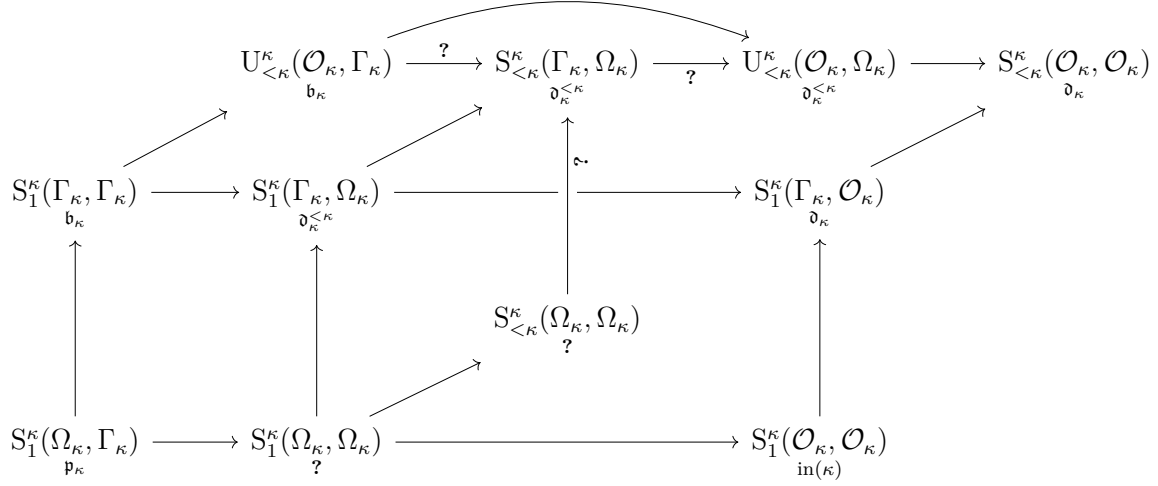

The unresolved questions indicated by the question marks in Figure~\ref{sdu} are formulated as the following open problems.

\begin{Prob}
\label{prob1}
Does $\Uk(\mathcal{O}_\kappa, \Gamma_\kappa)$ imply $\Sk(\Gamma_\kappa, \Omega_\kappa)$?
\end{Prob}

\begin{Prob}
\label{prob3}
Does $\Uk(\mathcal{O}_\kappa, \Omega_\kappa)$ imply $\Sk(\Gamma_\kappa, \Omega_\kappa)$?
\end{Prob}

\begin{Prob}
\label{prob4}
Does $\Sk(\Gamma_\kappa, \Omega_\kappa)$ imply $\Sk(\Omega_\kappa, \Omega_\kappa)$?
\end{Prob}

\begin{Prob}
\label{prob5}
$\non(\Sk(\Omega_\kappa, \Omega_\kappa)) =~?$
\end{Prob}

\begin{Prob}
\label{prob6}
$\non(\S1k(\Omega_\kappa, \Omega_\kappa)) =~?$
\end{Prob}

\section{$\kappa$-Hurewicz Conjecture}
\label{sec-hc}
We begin this section with the following basic implication diagram for arbitrary spaces (Figure~\ref{bd}), which follows from the results discussed in Section~\ref{sec-ie}.

\begin{figure}[h]
\begin{adjustbox}{max width = \textwidth, max height = \textheight, keepaspectratio, center}
\begin{tikzcd}[column sep=4ex, row sep=4ex]
\text{$K_\kappa$} \arrow[r, "\text{$\kappa$ is regular}"] & [1.2cm] \text{$\kappa$-Hurewicz} \arrow[r, "\text{$\kappa$ is regular}"] & [1.2cm] \text{$\kappa$-Scheepers} \arrow[r] & \text{$\kappa$-Menger} & \text{$\kappa$-Rothberger} \arrow[l]
\end{tikzcd}
\end{adjustbox}
\caption{Basic implication diagram}
\label{bd}
\end{figure}

Here we discuss the reverse implication of $K_\kappa \rightarrow \kappa\text{-Hurewicz}$, since the reverse implications corresponding to the other arrows have already been discussed in Section~\ref{sec-ni}. The following example shows that the reverse implication of $K_\kappa \rightarrow \kappa\text{-Hurewicz}$ fails for Tychonoff spaces.

\begin{Ex}
\label{ex1}
Assume that $\kappa$ is regular. Then there exists a Tychonoff $\kappa$-Hurewicz space which is not $K_\kappa$.
\end{Ex}

\begin{proof}
Let $L = [0,\kappa^+)$ be endowed with the order topology and $D = \{d_\xi : \xi< \kappa\}$ be a discrete space. Consider the topological sum $X = L\oplus D$. We first show that $X$ is $\kappa$-Hurewicz. We shall use the following observation.
\smallskip

\noindent\textbf{Claim.} Every open cover of $L$ of size at most $\kappa$ has a finite subcover.
\smallskip

Let $\mathcal U$ be an open cover of $L$ with $|\mathcal{U}| \leq \kappa$. For each $\beta< \kappa^+$, the ordinal interval $[0,\beta]$ is compact. Hence there exists $\mathcal{F}_\beta \in [\mathcal{U}]^{<\omega}$ such that $[0,\beta] \subseteq \bigcup \mathcal{F}_\beta$. For each $\mathcal{F} \in [\mathcal{U}]^{<\omega}$, put $B_{\mathcal{F}} = \{\beta< \kappa^+ : [0,\beta] \subseteq \bigcup \mathcal{F}\}$. Then \[\kappa^+ = \bigcup_{\mathcal{F}\in [\mathcal{U}]^{<\omega}} B_{\mathcal{F}}.\] Since $|[\mathcal U]^{<\omega}| \leq \kappa$ and $\cf(\kappa^+) = \kappa^+ > \kappa$, at least one of the sets $B_{\mathcal{F}}$ is cofinal in $\kappa^+$. Fix such a family $\mathcal{F} \in [\mathcal{U}]^{<\omega}$. Let $\gamma < \kappa^+$. Then there exists $\beta \in B_{\mathcal{F}}$ such that $\gamma \leq \beta$, and so $\gamma \in [0,\beta] \subseteq \bigcup \mathcal{F}$. Thus $L \subseteq \bigcup \mathcal{F}$, which proves the claim.
\smallskip

Now let $(\mathcal{U}_\alpha : \alpha < \kappa)\in \mathcal{O}_\kappa(X)^\kappa$ be such that no $\mathcal{U}_\alpha$ has a subcover of size less than $\kappa$. For each $\alpha< \kappa$, the family $\mathcal{U}_\alpha \restriction L = \{U\cap L : U\in \mathcal{U}_\alpha\}$ is an open cover of $L$ of size at most $\kappa$. By the claim, there exists $\mathcal{F}_\alpha \in [\mathcal{U}_\alpha]^{<\omega}$ such that $L\subseteq \bigcup \mathcal{F}_\alpha$. For every $\xi\leq \alpha$, choose $U_\xi^{(\alpha)} \in \mathcal{U}_\alpha$ such that $d_\xi \in U_\xi^{(\alpha)}$, and put $\mathcal{V}_\alpha = \mathcal{F}_\alpha \cup \{U_\xi^{(\alpha)} : \xi\leq \alpha\}$. Since $\alpha< \kappa$, we have $\mathcal{V}_\alpha \in [\mathcal{U}_\alpha]^{<\kappa}$. Set $V_\alpha = \bigcup \mathcal{V}_\alpha$ for every $\alpha< \kappa$. Then $\{V_\alpha : \alpha< \kappa\}\in \Gamma_\kappa(X)$. Therefore, $X$ is $\kappa$-Hurewicz.

We next show that $X$ is not $K_\kappa$. We first observe that no $\kappa$-compact subspace of $L$ is confinal in $\kappa^+$. Suppose that $A\subseteq L$ is confinal in $\kappa^+$. Consider $\mathcal{W} = \{A\cap [0,\beta) : \beta< \kappa^+\}$. Then $\mathcal{W}$ is an open cover of $A$. Let $\mathcal{W}_0 \in [\mathcal{W}]^{<\kappa}$. There is a set $B\in [\kappa^+]^{<\kappa}$ such that $\mathcal{W}_0 = \{A\cap [0,\beta) : \beta\in B\}$. Put $\gamma = \sup B$. Since $|B|< \kappa< \kappa^+ = \cf(\kappa^+)$, we have $\gamma< \kappa^+$. Therefore, $\bigcup \mathcal{W}_0 = A\cap [0,\gamma)$. Since $A$ is cofinal in $\kappa^+$ and $\gamma< \kappa^+$, there exists $a\in A$ such that $a\geq \gamma$, and hence $a\notin \bigcup \mathcal{W}_0$. Therefore, $\mathcal{W}_0$ does not cover $A$. Thus $\mathcal{W}$ has no subcover of size less than $\kappa$, and hence $A$ is not $\kappa$-compact.

Suppose, towards a contradiction, that $X$ is $K_\kappa$. Then $X = \bigcup_{\xi< \kappa} K_\xi$, where each $K_\xi$ is $\kappa$-compact. For each $\xi< \kappa$, there exists $\beta_\xi < \kappa^+$ such that $K_\xi \cap L \subseteq [0,\beta_\xi)$. It follows that \[L = \bigcup_{\xi< \kappa} (K_\xi \cap L) \subseteq [0, \sup_{\xi< \kappa} \beta_\xi).\] Since $\cf(\kappa^+) = \kappa^+ > \kappa$, we have $\beta = \sup_{\xi< \kappa} \beta_\xi < \kappa^+$. Consequently, $L \subseteq [0,\beta)$, which is a contradiction. Therefore, $X$ is not $K_\kappa$.
\end{proof}

In Example~\ref{ex1}, we have already observed that for a regular cardinal $\kappa$, a $\kappa$-Hurewicz space need not be $K_\kappa$, even in the class of Tychonoff spaces. However, for subspaces of $P(\kappa)$ and under the stronger assumption that $\kappa$ is weakly compact, the situation is less clear. In \cite{ucsp}, the authors formulated the following uncountable analogue of Hurewicz's Conjecture.

\begin{Conj}[$\kappa$-Hurewicz Conjecture]
Assume that $\kappa$ is weakly compact. Every $\kappa$-Hurewicz subspace of $P(\kappa)$ is a $K_\kappa$-space.
\end{Conj}

They also posed the following consistency problem concerning this conjecture.

\begin{Prob}[{\cite[Problem~8.1]{ucsp}}]
\label{prob7}
Is the $\kappa$-Hurewicz Conjecture consistent with ZFC + $\kappa$ is weakly compact?
\end{Prob}

Problem~\ref{prob7} asks whether there is a model of ZFC + $\kappa$ is weakly compact in which the $\kappa$-Hurewicz Conjecture holds; that is, in which every $\kappa$-Hurewicz subspace of $P(\kappa)$ is a $K_\kappa$-space. We show that if $\kappa$ is weakly compact and $2^{\bk} > 2^\kappa$, then there exists a $\kappa$-Hurewicz subspace of $P(\kappa)$ which is not $K_\kappa$; see Theorem~\ref{thm1}. Since the inequality $2^{\bk} > 2^\kappa$ is consistent with ZFC + $\kappa$ is weakly compact, it follows that there is a model of ZFC + $\kappa$ is weakly compact in which the $\kappa$-Hurewicz Conjecture fails. The proof of Theorem~\ref{thm1} relies on a few lemmas, which we establish first.

\begin{Def}[{\cite{ucsp}}]
\label{def1}
A mapping $f\in \kappa^\kappa$ is said to be
\begin{enumerate}[wide=0pt, label={\upshape(\arabic*)}, leftmargin=*]
\item strictly increasing if for all $\alpha< \beta< \kappa$, we have $f(\alpha)< f(\beta)$;

\item continuous if for any limit ordinal $\alpha< \kappa$, we have $f(\alpha) = \sup_{\beta< \alpha} f(\beta)$;

\item normal if it is strictly increasing and continuous.
\end{enumerate}
\end{Def}

Also recall from \cite{ucsp} that if $f \in \kappa^\kappa$ is normal, then \[\bigcup_{\alpha< \kappa} [f(\alpha), f(\alpha+1)) = [f(0), \kappa).\]

\begin{Lemma}[{\cite[Lemma~2.3]{ucsp}}]
\label{lemma1}
Assume that $\kappa$ is strongly inaccessible. Let $\fink \subseteq X\subseteq P(\kappa)$. If $\mathcal{U}$ is a $\kappa$-cover of $X$ by open sets in $P(\kappa)$, then there are a normal function $a\in \kappa^\kappa$ and a sequence $(U_\alpha : \alpha< \kappa)$ of pairwise distinct sets from $\mathcal{U}$ such that for every $x\in P(\kappa)$ and every $\alpha< \kappa$, $x \in U_\alpha$ whenever $x\cap [a(\alpha), a(\alpha+1)) = \emptyset$.
\end{Lemma}

We now observe that the proof of \cite[Theorem~3.6]{ucsp} yields a slightly stronger statement.

\begin{Lemma}
\label{lemma5}
Assume that $\kappa$ is strongly inaccessible. Let $X = \{x_\alpha : \alpha< \bk\}$ be a $\bk$-scale. Then for every $Y\subseteq X$, the space $Y\cup \fink$ is $\kappa$-Hurewicz.
\end{Lemma}

\begin{proof}
We modify the argument of \cite[Theorem~3.6]{ucsp}. The tail of the scale is treated as in that proof, while the initial part is handled directly using the definition of $\bk$.

By Theorem~\ref{thm5}, it is enough to show that $Y\cup \fink$ satisfies $\Uk(\Omega_\kappa, \Gamma_\kappa)$. Let $(\mathcal{U}_\xi : \xi< \kappa)$ be a sequence of $\kappa$-covers of $Y\cup \fink$ of size $\kappa$ by open sets in $P(\kappa)$ such that no $\mathcal{U}_\xi$ has a subcover of size less than $\kappa$. For every $\xi< \kappa$, Lemma~\ref{lemma1} provides a normal function $a_\xi \in \kappa^\kappa$ and a sequence $(U_\eta^{(\xi)} : \eta< \kappa)$ of pairwise distinct members of $\mathcal{U}_\xi$ such that for every $x\in P(\kappa)$ and every $\eta< \kappa$, $x \in U_\eta^{(\xi)}$ whenever $x\cap [a_\xi(\eta), a_\xi(\eta+1)) = \emptyset$. Define $c(\xi) = a_\xi(|\xi|^+)$ for all $\xi< \kappa$. Since $X$ is unbounded, there is an $\alpha< \bk$ such that $x_\alpha \not \leq_\kappa^* c$. Consequently, the set $I = \{\xi< \kappa : c(\xi)< x_\alpha(\xi)\}$ is of size $\kappa$. Put $A = Y\cap \{x_\beta : \beta< \alpha\}$. Then $|A|\leq |\alpha|< \bk$. For each $\xi\in I$, fix an enumeration $\mathcal{U}_\xi = \{W_\eta^{(\xi)} : \eta< \kappa\}$. For every $x\in A$, define $f_x\in \kappa^I$ by \[f_x(\xi) = \min \{\eta< \kappa : x\in W_\eta^{(\xi)}\}\] for all $\xi\in I$. Since $|I| = \kappa$, we may identify $I$ with $\kappa$. As the family $\{f_x : x\in A\}$ has size less than $\bk$, it is bounded. Hence there exists $g\in \kappa^I$ such that $f_x \leq_\kappa^* g$ for every $x\in A$.

For each $\xi< \kappa$, define
\[\mathcal{V}_\xi = \begin{cases}
\{W_\eta^{(\xi)} : \eta \leq g(\xi)\} \cup \{U_\eta^{(\xi)} : \eta \leq |\xi|^+\}, & \mbox{if } \xi\in I\\
\emptyset, & \mbox{if } \xi \notin I.
\end{cases}\]
Since $\kappa$ is strongly inaccessible, we have $|\xi|^+ < \kappa$ for every $\xi< \kappa$. Thus $\mathcal{V}_\xi \in [\mathcal{U}_\xi]^{<\kappa}$ for every $\xi< \kappa$. We show that every point of $Y\cup \fink$ belongs to $\bigcup \mathcal{V}_\xi$ for all but less than $\kappa$ many $\xi\in I$. First, let $x\in A$. Since $f_x \leq_\kappa^* g$, we have $f_x(\xi) \leq g(\xi)$ for all but less than $\kappa$ many $\xi\in I$. Hence $x\in W_{f_x(\xi)}^{(\xi)} \subseteq \bigcup \mathcal{V}_\xi$.

Next, let $x\in \fink$. Since $|x|< \kappa$ and $\kappa$ is regular, there exists $\gamma< \kappa$ such that $x\subseteq \gamma$. Also since $a_\xi$ is strictly increasing for every $\xi< \kappa$, we get $x\cap [a_\xi(\xi), a_\xi(\xi+1)) = \emptyset$ for all sufficiently large $\xi\in I$. It follows that $x\in U_\xi^{(\xi)} \subseteq \bigcup \mathcal{V}_\xi$ as $\xi< |\xi|^+$.

Finally, let $x_\beta\in Y$ with $\beta \geq \alpha$. Since $X$ is a $\bk$-scale, $x_\alpha \leq_\kappa^* x_\beta$. Thus for all but less than $\kappa$ many $\xi\in I$, $x_\beta(\xi) \geq x_\alpha(\xi) > c(\xi) = a_\xi(|\xi|^+)$. It follows that $|x_\beta \cap [0, a_\xi(|\xi|^+))| < |\xi|^+$ since $x_\beta$ is identified with the increasing enumeration of its elements. As $a_\xi(0) = 0$, $a_\xi$ is continuous, and $|\xi|^+$ is a limit ordinal, we have $a_\xi(|\xi|^+) = \sup_{\eta< |\xi|^+} a_\xi(\eta)$. Therefore, $[0, a_\xi(|\xi|^+)) = \bigcup_{\eta< |\xi|^+} [a_\xi(\eta), a_\xi(\eta+1))$. The intervals on the right are pairwise disjoint, and there are $|\xi|^+$ many of them. Hence there exists $\eta< |\xi|^+$ such that $x_\beta \cap [a_\xi(\eta), a_\xi(\eta+1)) = \emptyset$. Thus $x_\beta \in U_\eta^{(\xi)} \subseteq \bigcup \mathcal{V}_\xi$.

We have therefore shown that for each $x\in Y\cup \fink$, the set $\{\xi\in I : x\notin \bigcup \mathcal{V}_\xi\}$ has size less than $\kappa$. It follows that $\{V\in \{\bigcup \mathcal{V}_\xi : \xi< \kappa\} : x\notin V\}$ has size less than $\kappa$. It remains to verify that $\{\bigcup \mathcal{V}_\xi : \xi< \kappa\}$ is of size $\kappa$. For every $\xi\in I$, $\bigcup \mathcal{V}_\xi \neq Y\cup \fink$. Suppose, to the contrary, that $\{\bigcup \mathcal{V}_\xi : \xi\in I\}$ has size less than $\kappa$. For each distinct set $V = \bigcup \mathcal{V}_\xi$, $\xi\in I$, choose $x_V\in (Y\cup \fink) \setminus V$. Clearly, $\{\xi\in I : \bigcup \mathcal{V}_\xi = V\} \subseteq \{\xi\in I : x_V \notin \bigcup \mathcal{V}_\xi\}$. It follows that $\{\xi\in I : \bigcup \mathcal{V}_\xi = V\}$ has size less than $\kappa$. Since $\kappa$ is regular, less than $\kappa$ many such sets cannot cover $I$, contradicting $|I| = \kappa$. Hence $\{\bigcup \mathcal{V}_\xi : \xi< \kappa\}$ is of size $\kappa$. Therefore, $\{\bigcup \mathcal{V}_\xi : \xi< \kappa\}$ is $\gamma_\kappa$-cover of $Y\cup \fink$ by open sets in $P(\kappa)$. Thus $Y\cup \fink$ satisfies $\Uk(\Omega_\kappa, \Gamma_\kappa)$.
\end{proof}

Thus a single $\bk$-scale produces a large collection of $\kappa$-Hurewicz spaces. Indeed, since $X\cap \fink = \emptyset$, the sets $Y\cup \fink$ and $Y\subseteq X$ are pairwise distinct. Consequently, there are at least $2^{\bk}$ distinct $\kappa$-Hurewicz subspaces of $P(\kappa)$ of this form. We next estimate the number of $K_\kappa$-subspaces of $P(\kappa)$ when $\kappa$ is weakly compact.

\begin{Lemma}
\label{lemma6}
If $\kappa$ is weakly compact, then the number of $K_\kappa$-subspaces of
$P(\kappa)$ is at most $2^\kappa$.
\end{Lemma}

\begin{Th}
\label{thm1}
Assume that $\kappa$ is weakly compact. If $2^{\bk} > 2^\kappa$, then there exists a $\kappa$-Hurewicz subspace of $P(\kappa)$ which is not $K_\kappa$.
\end{Th}

\begin{proof}
Let $X = \{x_\alpha : \alpha< \bk\}$ be a $\bk$-scale. By Lemma~\ref{lemma5}, for every $Y\subseteq X$, the space $Y\cup \fink$ is $\kappa$-Hurewicz. Since $X\cap \fink = \emptyset$, the map $Y\mapsto Y\cup \fink$ is injective. Hence there are $2^{\bk}$ distinct $\kappa$-Hurewicz subspaces of $P(\kappa)$ of this form.

On the other hand, by Lemma~\ref{lemma6}, there are at most $2^\kappa$ $K_\kappa$-subspaces of $P(\kappa)$. Since $2^{\bk} > 2^\kappa$, there exists some $Y\subseteq X$ such that $Y\cup \fink$ is $\kappa$-Hurewicz but not $K_\kappa$.
\end{proof}

In particular, the $\kappa$-Hurewicz Conjecture cannot hold under GCH at $\kappa$.

\begin{Cor}
\label{cor2}
Assume that $\kappa$ is weakly compact, and assume GCH at $\kappa$. Then there exists a $\kappa$-Hurewicz subspace of $P(\kappa)$ which is not $K_\kappa$.
\end{Cor}

The preceding results do not settle Problem~\ref{prob7}, but they show that any model in which $\kappa$ is weakly compact and the $\kappa$-Hurewicz Conjecture holds must satisfy $2^{\bk} = 2^\kappa$. In particular, such a model cannot satisfy $2^\kappa = \kappa^+$.

\section{On further problems posed in \cite{ucsp}}
\label{sec-fp}
Finally, we consider the following problems posed in \cite{ucsp}.

\begin{Prob}[{\cite[Problem~8.2]{ucsp}}]
\label{prob8}
Is it consistent with ZFC + $\kappa$ is weakly compact that there is a $\bk$-scale set which is $K_\kappa$?
\end{Prob}

\begin{Prob}[{\cite[Problem~8.3]{ucsp}}]
\label{prob10}
Is there an alternative construction of a $\kappa$-Menger space which is not a $K_\kappa$-space, different from the construction used in the solution of the $\kappa$-Hurewicz Problem?
\end{Prob}

Problem~\ref{prob8} asks whether there is a model of ZFC + $\kappa$ is weakly compact in which there exists a $\bk$-scale $X$ such that the corresponding scale set $X\cup \fink$ is $K_\kappa$. We first point out that Theorem~\ref{thm1} does not settle Problem~\ref{prob8}. Under the assumption $2^{\bk} > 2^\kappa$, the counting argument yields a subset $Y$ of a $\bk$-scale $X$ such that $Y\cup \fink$ is $\kappa$-Hurewicz but not $K_\kappa$. However, this does not imply that the full $\bk$-scale set $X\cup \fink$ fails to be $K_\kappa$. Consequently, Theorem~\ref{thm1} does not imply the existence of a model of ZFC + $\kappa$ is weakly compact in which Problem~\ref{prob8} has a negative answer.

We reformulate Problem~\ref{prob8} as Problem~\ref{prob2} and then derive a necessary tree-theoretic condition for a positive answer; see Theorem~\ref{thm13}. As a consequence, under the additional assumption of the Perfect Subtree Property \cite{hm23}, we obtain a negative answer to Problem~\ref{prob8}; the question under weak compactness alone remains open. To this end, we begin by recording a characterization of the $K_\kappa$ property for sets of the form $X\cup \fink$, where $X\subseteq [\kappa]^\kappa$.

\begin{Lemma}
\label{lemma16}
Let $\kappa$ be a weakly compact cardinal, and let $X\subseteq [\kappa]^\kappa$. Then the following assertions are equivalent.
\begin{enumerate}[wide=0pt,label={\upshape(\arabic*)},leftmargin=*]
\item $X\cup \fink$ is a $K_\kappa$-space.
\item $X$ is a union of $\kappa$ many closed subsets of $[\kappa]^\kappa$.
\end{enumerate}
\end{Lemma}

\begin{proof}
Assume first that $X\cup \fink = \bigcup_{\xi< \kappa} K_\xi$, where each $K_\xi$ is $\kappa$-compact. Then every $K_\xi$ is closed in $P(\kappa)$. Consequently, $C_\xi = K_\xi \cap [\kappa]^\kappa$ is closed in $[\kappa]^\kappa$ for every $\xi< \kappa$, and $X = \bigcup_{\xi< \kappa} C_\xi$.

Conversely, suppose that $X = \bigcup_{\xi< \kappa} C_\xi$, where every $C_\xi$ is closed in $[\kappa]^\kappa$. For each $\xi< \kappa$, let $F_\xi = \overline{C_\xi}^{P(\kappa)}$. Then $F_\xi \cap [\kappa]^\kappa = C_\xi$, and so $F_\xi \subseteq C_\xi \cup \fink \subseteq X\cup \fink$. Moreover, each $F_\xi$ is $\kappa$-compact. Also since $\kappa$ is regular, no member of $\fink$ is cofinal in $\kappa$, and hence $\fink = \bigcup_{\alpha< \kappa} P(\alpha)$. For every $\alpha< \kappa$, the set $P(\alpha)$ is $\kappa$-compact. It follows that $X\cup \fink = \bigcup_{\xi< \kappa} F_\xi \cup \bigcup_{\alpha< \kappa} P(\alpha)$ is a union of $\kappa$ many $\kappa$-compact subspaces. Thus $X\cup \fink$ is a $K_\kappa$-space.
\end{proof}

Lemma~\ref{lemma16} reformulates Problem~\ref{prob8} as follows.

\begin{Prob}[{\cite[Problem~8.2]{ucsp}}]
\label{prob2}
Is it consistent with ZFC + $\kappa$ is weakly compact that there exists a $\bk$-scale which is a union of $\kappa$ many closed subsets of $[\kappa]^\kappa$?
\end{Prob}

Let $\lambda$ be an infinite cardinal. Recall that a set $X\subseteq P(\kappa)$ is said to be $\lambda$-concentrated on a set $Q\subseteq X$ if for each open set $U\subseteq P(\kappa)$ containing $Q$, we have $|X\setminus U|< \lambda$.

\begin{Lemma}[{\cite[Lemma~3.8]{ucsp}}]
\label{lemma17}
Assume that $\kappa$ is weakly compact. If $X$ is a $\bk$-scale, then $X\cup \fink$ is $\bk$-concentrated on $\fink$.
\end{Lemma}

\begin{Lemma}
\label{lemma18}
Assume that $\kappa$ is weakly compact. If $X$ is a $\bk$-scale, then $X\cup \fink$ contains no subspace homeomorphic to $2^\kappa$.
\end{Lemma}

\begin{proof}
Suppose, towards a contradiction, that $C\subseteq X\cup \fink$ is homeomorphic to $2^\kappa$. Let $\Phi: 2^\kappa \to C$ be a homeomorphism. Since $|\fink| = \kappa$, we have $|\Phi^{-1}[C\cap \fink]| \leq \kappa$. We claim that there is a closed set $D\subseteq 2^\kappa \setminus \Phi^{-1}[C\cap \fink]$ homeomorphic to $2^\kappa$. To see this, enumerate $\Phi^{-1}[C\cap \fink] = \{z_\xi : \xi< \lambda\}$, $\lambda \leq \kappa$. Choose a set $B\subseteq \kappa$ such that $|B| = |\kappa \setminus B| = \kappa$, and choose pairwise distinct coordinates $\gamma_\xi \in B$ for every $\xi< \lambda$. Define \[D = \left\{z\in 2^\kappa : z(\gamma_\xi) \neq z_\xi(\gamma_\xi) \text{ for every } \xi< \lambda\right\}.\] Then $D$ is closed. Moreover, $z_\xi \notin D$ for every $\xi< \lambda$, and hence $D\cap \Phi^{-1}[C\cap \fink] = \emptyset$. Since the coordinates in $\kappa \setminus B$ are unrestricted and $|\kappa \setminus B| = \kappa$, $D$ is homeomorphic to $2^\kappa$.

The set $\Phi[D]$ is homeomorphic to $2^\kappa$, and therefore is $\kappa$-compact. Thus $\Phi[D]$ is closed in $P(\kappa)$. Moreover, $\Phi[D] \cap \fink = \emptyset$. Consequently, $U = P(\kappa) \setminus \Phi[D]$ is an open set containing $\fink$. Since $X\cup \fink$ is $\bk$-concentrated on $\fink$ (see Lemma~\ref{lemma17}), $|(X\cup \fink) \setminus U|< \bk$. On the other hand, $\Phi[D] \subseteq (X\cup \fink) \setminus U$ and $|\Phi[D]| = 2^\kappa \geq \bk$, a contradiction. Therefore, $X\cup \fink$ contains no subspace homeomorphic to $2^\kappa$.
\end{proof}

We can now obtain a tree-theoretic obstruction to a positive answer to Problem~\ref{prob8}. Recall that a $\kappa$-tree is a tree of height $\kappa$ whose levels have size less than $\kappa$.

\begin{Th}
\label{thm13}
Assume that $\kappa$ is weakly compact. If $X = \{x_\alpha : \alpha< \bk\}$ is a $\bk$-scale such that $X\cup \fink$ is a $K_\kappa$-space, then there exists a $\kappa$-tree having exactly $\bk$ many cofinal branches and having no perfect subtree.
\end{Th}

\begin{proof}
By Lemma~\ref{lemma16}, there are closed subsets $C_\xi \subseteq [\kappa]^\kappa$, $\xi< \kappa$, such that $X = \bigcup_{\xi< \kappa} C_\xi$. Since $\bk$ is regular and $\kappa< \bk$, there exists $\xi< \kappa$ such that $|C_\xi| = \bk$. Fix such a $C_\xi$ and put $F = \overline{C_\xi}^{P(\kappa)}$. Then $F\subseteq C_\xi\cup Q_\kappa \subseteq X\cup Q_\kappa$ and  $|F| = \bk$. Associate with $F$ the tree \[T_F = \left\{x\restriction \alpha : x\in F,\ \alpha< \kappa\right\} \subseteq 2^{<\kappa},\] where the members of $P(\kappa)$ are identified with their characteristic functions. Since $\kappa$ is strongly inaccessible, $|(T_F)_\alpha| \leq 2^{|\alpha|} < \kappa$ for every $\alpha< \kappa$. Hence $T_F$ is a $\kappa$-tree. Since $F$ is closed in $2^\kappa$, the set \[[T_F] = \{x\in 2^\kappa : x\restriction \alpha \in T_F \text{ for every } \alpha< \kappa\}\] of cofinal branches through $T_F$ is precisely $F$. Thus $|[T_F]| = |F| = \bk$.

We claim that $T_F$ has no perfect subtree. Suppose, to the contrary, that $S\subseteq T_F$ is a perfect subtree. Using the splitting property of $S$, recursively choose $\{s_u : u\in 2^{<\kappa}\} \subseteq S$ such that $u\subseteq v$ implies $s_u\subseteq s_v$, $s_{u^\frown 0} \perp s_{u^\frown 1}$ for every $u\in 2^{<\kappa}$, and $\dom(s_u) \geq \dom(u)$ for every $u\in 2^{<\kappa}$. At limit stages of length less than $\kappa$, the required nodes are obtained by taking unions, using the $<\kappa$-closure of $S$.

Define $\Phi: 2^\kappa \to [S]$ by $\Phi(x) = \bigcup_{\alpha< \kappa} s_{x\restriction \alpha}$ for all $x\in 2^\kappa$. Since $\Phi$ is injective, we may regard it as a bijection $\Phi: 2^\kappa \to \Phi[2^\kappa]$. For $u\in 2^{<\kappa}$, let $[u] = \{x\in 2^\kappa : u\subseteq x\}$. We have $\Phi[[u]] = [s_u] \cap \Phi[2^\kappa]$. Indeed, if $x\in [u]$, then $s_u\subseteq \Phi(x)$. Conversely, if $x\notin [u]$, let $\delta< \dom(u)$ be the first coordinate at which $x$ and $u$ differ. Then the nodes corresponding to $u\restriction (\delta+1)$ and $x\restriction (\delta+1)$ are incompatible. Hence $s_u \nsubseteq \Phi(x)$, and therefore $\Phi(x) \notin [s_u]$. It follows that $\Phi^{-1}: h\to 2^\kappa$ is continuous.

To see that $\Phi$ is continuous, let $x\in 2^\kappa$ and let $[\Phi(x)\restriction \beta]$, $\beta< \kappa$, be a basic neighbourhood of $\Phi(x)$. Since the domains of the nodes $s_{x\restriction \alpha}$, $\alpha< \kappa$, are cofinal in $\kappa$, there exists $\alpha< \kappa$ such that $\dom(s_{x\restriction \alpha}) \geq \beta$. Whenever $y\restriction \alpha = x\restriction \alpha$, we have $s_{x\restriction \alpha} \subseteq \Phi(y)$, and consequently $\Phi(y)\restriction \beta = \Phi(x)\restriction \beta$. Therefore, $\Phi[[x\restriction \alpha]] \subseteq [\Phi(x)\restriction \beta]$. Thus $\Phi$ is continuous. Hence $\Phi$ is a homeomorphism of $2^\kappa$ onto $\Phi[2^\kappa] \subseteq [S]$. Therefore, $[S]$ contains a subspace homeomorphic to $2^\kappa$. Since $S\subseteq T_F$, $[S] \subseteq [T_F] = F$. Thus $F$ contains a subspace homeomorphic to $2^\kappa$. But $F\subseteq X\cup \fink$, contradicting Lemma~\ref{lemma18}. Therefore, $T_F$ has no perfect subtree.
\end{proof}

The preceding theorem connects Problem~\ref{prob8} with the Perfect Subtree Property. Recall that the Perfect Subtree Property for $\kappa$ \cite{hm23} asserts that every $\kappa$-tree with at least $\kappa^+$ many cofinal branches has a perfect subtree.

\begin{Cor}
\label{cor9}
Assume that $\kappa$ is weakly compact and satisfies the Perfect Subtree Property. Then no $\bk$-scale set is a $K_\kappa$-space.
\end{Cor}

\begin{proof}
Suppose that $X$ is a $\bk$-scale and that $X\cup \fink$ is a $K_\kappa$-space. By Theorem~\ref{thm13}, there is a $\kappa$-tree with $\bk$ many cofinal branches and no perfect subtree. Since $\bk> \kappa$, the tree has at least $\kappa^+$ many cofinal branches. This contradicts the Perfect Subtree Property.
\end{proof}

Theorem~\ref{thm13} yields a necessary condition for a positive answer to Problem~\ref{prob8}: in any model in which a $\bk$-scale set is a $K_\kappa$-space, there must exist a $\kappa$-tree with exactly $\bk$ many cofinal branches and no perfect subtree.

\begin{Rem}
\label{rem6}
The additional hypothesis in Corollary~\ref{cor9} is substantial. Hayut and M\"{u}ller \cite{hm23} studied the Perfect Subtree Property at weakly compact cardinals and showed that it has considerable consistency strength.
\end{Rem}

\begin{Cor}
\label{cor10}
Assume that $\kappa$ is weakly compact and satisfies the Perfect Subtree Property. Then there exists a $\kappa$-Menger subspace of $2^\kappa$ which is not a $K_\kappa$-space. In fact, every $\mathfrak b_\kappa$-scale set has this property.
\end{Cor}

\begin{proof}
Let $X\subseteq [\kappa]^\kappa$ be a $\bk$-scale. By \cite[Theorem~3.6]{ucsp}, the $\bk$-scale set $X\cup \fink$ is $\kappa$-Hurewicz, and hence it is $\kappa$-Menger. On the other hand, by Corollary~\ref{cor9}, the Perfect Subtree Property implies that no $\bk$-scale set is a $K_\kappa$-space. Consequently, $X\cup \fink$ is $\kappa$-Menger but not $K_\kappa$.
\end{proof}

Thus under the Perfect Subtree Property, every $\bk$-scale set provides an alternative example of a $\kappa$-Menger space which is not a $K_\kappa$-space. In particular, this gives a positive answer to Problem~\ref{prob10} under this additional assumption. Notice that no new scale set is constructed here; rather, the Perfect Subtree Property ensures that every $\bk$-scale set fails to be a $K_\kappa$-space. More importantly, the examples obtained in this way are $\kappa$-Hurewicz, whereas the example constructed in \cite[Theorem~4.4]{ucsp} is $\kappa$-Menger but not $\kappa$-Hurewicz. Thus the failure of the $K_\kappa$ property in Corollary~\ref{cor10} does not follow from a failure of the $\kappa$-Hurewicz property.

%
%
%
%

\end{document}